\documentclass{article}
\usepackage[utf8]{inputenc}
\usepackage{amsmath,graphicx,amssymb, amsfonts, amsthm,mathtools,fullpage,cite}
\usepackage{subfigure}
\usepackage[linesnumbered,algoruled,lined]{algorithm2e}
\newcommand\numberthis{\addtocounter{equation}{1}\tag{\theequation}}
\usepackage[british]{babel}
\usepackage{color}
\usepackage{bm}
\usepackage{mathrsfs}
\newtheorem{theorem}{Theorem}[section]
\newtheorem{corollary}{Corollary}[section]
\newtheorem{lemma}{Lemma}[section]
\newtheorem{defn}{Definition}[section]

\newcommand{\T}{{\mathcal T}}
\renewcommand{\S}{{\mathbb S}}
\newcommand{\U}{\bm{U}}
\newcommand{\V}{\bm{V}}
\renewcommand{\P}{\mathscr{P}}
\newcommand{\Pui}{\mathscr{P}_{\U^{(i)}}^{(i)}}
\newcommand{\Puil}{\mathscr{P}_{\U^{(i)}_l}^{(i)}}
\newcommand{\Puilp}{\mathscr{P}_{\U^{{(i)}\perp}_l}^{(i)}}
\newcommand{\Pujnil}{\mathscr{P}^{(j\neq i)}_{\C_l,\{\U_l^{(j)}\}_{j\neq i}}}
\newcommand{\R}{{\mathbb R}}
\newcommand{\Z}{{\mathcal Z}}
\newcommand{\C}{\mathcal{C}}
\newcommand{\I}{{\mathscr I}}

\newcommand{\inner}[1]{\langle #1 \rangle}

\newcommand{\A}{{\mathscr A}}

\newcommand{\rr}{{\bm r}}
\newcommand{\norm}[1]{\left\lVert#1\right\rVert}
\newcommand{\abs}[1]{\left\lvert#1\right\rvert}

\title{Provable Near-Optimal Low-Multilinear-Rank Tensor Recovery}
\author{Jian-Feng Cai\thanks{Department of Mathematics, the Hong Kong University of Science and Technology, Clear Water Bay, Kowloon, Hong Kong SAR, China. Emails: \texttt{jfcai, lmiao, yangwang, mayxian@ust.hk}}
\and
Lizhang Miao$^*$\thanks{Corresponding author}
\and
Yang Wang$^*$
\and
Yin Xian$^*$
}

\begin{document}

\maketitle

\begin{abstract}
We consider the problem of recovering a low-multilinear-rank tensor from a small amount of linear measurements. We show that the Riemannian gradient algorithm initialized by one step of iterative hard thresholding can reconstruct an order-$d$ tensor of size $n\times\ldots\times n$ and multilinear rank $(r,\ldots,r)$ with high probability from only $O(nr^2 + r^{d+1})$ measurements, assuming $d$ is a constant. This sampling complexity is optimal in $n$, compared to existing results whose sampling complexities are all unnecessarily large in $n$. The analysis relies on the tensor restricted isometry property (TRIP) and the geometry of the manifold of all tensors with a fixed multilinear rank. High computational efficiency of our algorithm is also achieved by doing higher order singular value decomposition on intermediate small tensors of size only $2r\times \ldots\times 2r$ rather than on tensors of size $n\times \ldots\times n$ as usual. 
\end{abstract}

\section{Introduction}
The tensor recovery problem arises in a variety of applications, such as machine learning~\cite{argyriou2007multi,collins2012tensor, romera2013multilinear,anandkumar2014tensor}, signal processing~\cite{mesgarani2006discrimination,li2010tensor,nion2010tensor}, bioinformatics~\cite{troyanskaya2001missing}, and quantum state tomography~\cite{gross2010quantum,gross2011recovering}. Let $\T \in \R^{n_1\times n_2\times\cdots\times n_d}$ be an unknown tensor. The goal of tensor recovery problem is to reconstruct $\T$ from its linear measurements 
\begin{equation}\label{eq:mea}
\bm{y}=\A\T,
\end{equation}
where $\A~:~\mathbb{R}^{n_1\times n_2\times\cdots\times n_d}\to \R^m$ with $m\ll \prod_{i=1}^d n_i$ is a linear operator defined by
\begin{equation}\label{eq:meaA}
[\A\T]_i=\langle \mathcal{A}_i,\T \rangle,\quad i=1,\ldots,m
\end{equation}
with $\mathcal{A}_i\in\R^{n_1\times n_2\times\cdots\times n_d}$ for $i=1,\ldots,m$ measurement tensors. Since $m\ll \prod_{i=1}^d n_i$, it is impossible to have a unique tensor recovery if no additional structure on $\T$ is assumed. The additional structure is to make the tensor compressible, meaning only very few parameters are able to determine the tensor completely. Similar to a matrix, a popular way to describe the compressibility of a tensor is its rank. As long as the rank of $\mathcal{T}$ is low enough such that the degree of freedom in $\T$ is small enough compared to $m$, it is possible to recover $\T$ from its linear measurements $\bm{y}$ defined in \eqref{eq:mea}.

As a special case of low-rank tensor recovery where $d=2$, low-rank matrix recovery have been investigated extensively since the pioneering works \cite{fazel2002matrix,candes2009exact,candes2009power}. The low-rank matrix recovery problem is NP-hard \cite{harvey2006complexity} in the worst case. Nevertheless, there are many algorithms available for successful low-rank matrix recovery in most of the cases. They can be categorized into convex and non-convex optimization based approaches. In convex optimization based approaches \cite{fazel2002matrix,candes2009exact,candes2009power,cai2010singular}, nuclear norm minimizations are usually applied to recover the low-rank matrix from its linear measurements. Under suitable assumptions and different settings, it is shown that $m\sim O(nr\log^{\alpha}n)$ for some $\alpha\geq0$ is sufficient for an exact recovery of an $n\times n$ matrix of rank $r$ via a nuclear norm minimization. However, the computation of nuclear norm minimization can be expensive and consumes large memory, though the low-rank structure can be exploited \cite{cai2010singular}. In non-convex methods, the unknown low-rank matrix is either parameterized in a factorization form or represented as an element in the set of all low-rank matrices. Therefore, they are generally faster than their convex counterparts and use much less memory. A major difficulty of non-convex approaches is how to avoid possible local minima. Surprisingly, it has been shown that these non-convex low-rank matrix recovery approaches are guaranteed to converge to the global minimum \cite{jain2010guaranteed,tanner2013normalized,blanchard2015cgiht,wei2016guarantees,wei2020guarantees,cai2018solving,chi2019nonconvex}, and there is no spurious local minima of many non-convex functions for low-rank matrix recovery \cite{li2019toward,ge2017no,sun2018geometric,zhu2018global}. The sampling complexity is typically $O(nr^{\beta_1}\log^{\beta_2}n)$ with some $\beta_1\geq 2$ and $\beta_2\geq 0$.  

However, extending these low-rank matrix recovery approaches to tensors with $d\geq 3$ is not straightforward and sometimes challenging. 
Similar to matrix, researchers look for efficient ways to decompose high order tensor so that we encode tensor in low dimension. Different decompositions for high order ($d\ge3$) tensors lead to different definitions of rank, such as the CP-rank, the tubal rank, and the multilinear rank. 
\begin{itemize}
\item 
The CP-rank of $\mathcal{X}\in\R^{n_1\times n_2\times\cdots\times n_d}$ is the smallest number of rank one tensors that sum up to $\mathcal{X}$ ~\cite{kolda2009tensor}. More precisely, $\mathcal{X}$ is of CP-rank $r$ if it can be decomposed as in the following with the minimum possible integer $r$
$$
\mathcal{X}=\sum_{i=1}^{r}b_i\bm{v}^{(1)}_i\otimes\bm{v}^{(2)}_i\otimes\ldots\otimes\bm{v}^{(d)}_i,
$$
where $b_i\in\mathbb{R}$ are coefficients, $\bm{v}_i^{(j)}\in\mathbb{R}^{n_j}$ are unit vectors, and $\bm{v}^{(1)}\otimes\bm{v}^{(2)}\otimes\ldots\otimes\bm{v}^{(d)}\in\R^{n_1\times n_2\times\cdots\times n_d}$ is a tensor with $(k_1,\ldots,k_d)$-th entry $v^{(1)}_{k_1}v^{(2)}_{k_2}\ldots v^{(d)}_{k_d}$. The CP-rank is a natural generalization of matrix rank to tensor, and it fully exploits the multilinear structure of tensors along all $d$ directions. However, the low-CP-rank approximation of a given tensor is very ill-posed \cite{de2008tensor}, especially for the degenerate case. As a consequence, low-CP-rank tensor recovery problem is extremely difficult to solve. Existing works are either computationally intractable (e.g. \cite{yuan2016tensor}) or applicable to only special low-CP-rank tensors (e.g. \cite{shah2015optimal}). 
\item 
The tubal rank \cite{kilmer2013third} of a 3-D tensor $\mathcal{X}\in\R^{n_1\times n_2\times n_3}$ is the same as
$$
\mbox{Tubal-rank}(\mathcal{X})=\max_{1\leq i\leq n_3}\mathrm{rank}([\mathscr{F}_3\mathcal{X}]_{:,:,i}),
$$
where $\mathscr{F}_3~:~\R^{n_1\times n_2\times n_3}\to\mathbb{C}^{n_1\times n_2\times n_3}$ is the discrete Fourier transform along the third direction, $[\cdot]_{:,:,i}$ is the $i$-th slice matrix in the third direction, and $\mathrm{rank}(\cdot)$ is the matrix rank. It is essentially a matrix rank in the Fourier domain. Though the exact best low-tubal-rank approximation of a 3-D tensor can be done by t-SVD \cite{kilmer2013third} and low-tubal-rank tensor recovery has a nice theory \cite{lu2018exact}, the tubal rank does not take full advantage of the multi-linear structure --- it utilizes only the linear dependencies in the slice matrices and it defines only for 3-D tensors. In practice, it usually underperforms other low rank tensor models in terms of data representation efficiency.
\item  
The multilinear rank of $\mathcal{X}\in\R^{n_1\times n_2\times\cdots\times n_d}$ is a tuple of $d$ integers $\bm{r}=(r_1,\ldots,r_d)$ defined via the ranks of the matricizations of $\mathcal{X}$~\cite{de2000multilinear}, and its associated Tucker decomposition is
\begin{equation}\label{eq:Tucker1}
\mathcal{X}=\sum_{i_1=1}^{r_1}\ldots\sum_{i_d=1}^{r_d}b_{i_1\ldots i_d}\bm{v}^{(1)}_{i_1}\otimes\bm{v}^{(2)}_{i_2}\otimes\ldots\otimes\bm{v}^{(d)}_{i_d},
\end{equation}
where $\mathcal{B}=[b_{i_1\ldots i_d}]_{(i_1,\ldots, i_d)=(1,\ldots,1)}^{(r_1,\ldots,r_d)}\in\R^{r_1\times \ldots\times r_d}$ is the core tensor, and $\bm{V}^{(j)}=[\bm{v}^{(j)}_{1},\ldots,\bm{v}^{(j)}_{r_j}]\in\R^{n_j\times r_j}$, $j=1,\ldots,d$, are orthogonal matrices. Similar to the CP-rank decomposition, the Tucker multilinear decomposition fully explores the low dimensional structure along all directions. Furthermore, the low-multilinear-rank approximation can be computed via the truncated higher-order (HOSVD) \cite{de2000multilinear}. 
\end{itemize}

Comparing these different low-rank tensor models, the low-multilinear-rank tensor model has both advantages of the low-CP-rank and low-tubal-rank tensor models, and it avoids drawbacks of both models. On the one hand, the multilinear rank makes full use of the multilinear structure of the tensor as in the CP rank. On the other hand, the low-rank approximation under multilinear rank can be computed efficiently as in the tubal rank. Therefore, in this paper, we consider the low-rank tensor recovery problem \eqref{eq:mea} and \eqref{eq:meaA} under the assumption that the unknown tensor $\mathcal{T}$ has a low multilinear rank. In particular, denote $\mbox{mulrank}(\cdot)$ the multilinear rank of a tensor, and we solve
\begin{equation}\label{eq:tensorrecovery}
\mbox{Recover $\mathcal{T}$ from \eqref{eq:mea} and \eqref{eq:meaA} subject to $\mbox{mulrank}(\mathcal{T})=\bm{r}:=(r_1,r_2,\ldots,r_d)$.}
\end{equation}

Many provable approaches are developed and analyzed for solving the tensor recovery problem \eqref{eq:tensorrecovery}; see, e.g., \cite{mu2014square,huang2014provable,xia2017polynomial,yuan2016tensor}. 
Denote $n=\max\{n_1,\ldots,n_d\}$ and $r=\max\{r_1,\ldots,r_d\}$. Since the multilinear rank is the ranks of matricizations of a tensor, one can minimize the sum of matrix nuclear norms of matricizations to obtain a low multilinear rank, which is investigated in \cite{mu2014square,huang2014provable}.  Under different random sampling schemes, it was shown theoretically \cite{mu2014square,huang2014provable} that the sampling complexity (i.e., the minimum $m$ for a successful tensor recovery) is $O(n^{d-1}r)$, which is improved to $O(n^{\lceil d/2\rceil} r^{\lfloor d/2\rfloor})$ by square reshaping \cite{mu2014square}. These results are consistent with those when the low-rank tensor is unfolded as a low-rank matrix, due to the fact that these approaches are essentially matrix nuclear norm minimization. By developing a series of algebraic and probabilistic techniques, Yuan and Zhang~\cite{yuan2016tensor} proposed a tensor nuclear norm minimization for tensor completion, a special case of low-multilinear-rank tensor recovery where entries of the tensor are sampled. When $d=3$, it is proved that the tensor can be recovered exactly with high probability with entries as few as $O((n^{3/2}r^{1/2}+nr^2)\log^2 n)$. However, this tensor nuclear norm minimization is computationally intractable. Later, a non-convex approach \cite{xia2017polynomial} is developed to directly attack the tensor completion problem, and it shows that the tensor can be reconstructed from sampled entries as few as $O(n^{3/2}r^{7/2}\log^{7/2}n+nr^7\log^6n)$ with a polynomial computational complexity. A summary of existing provable low-multilinear-rank tensor recovery methods is shown in Table~\ref{table:tensor_m}, where for simplicity the sampling complexity is counted for tensors in $\R^{n\times n\times n}$ with rank $(r,r,r)$.

\begin{table}[ht]
\centering
\begin{tabular}{cccc}
\hline \hline
 & Optimization & Sampling Scheme & Sampling Complexity    \\
 \hline
  this paper & Non-Convex & Gaussian measurement &  $O(nr^2+r^4)$ \\
  \cite{mu2014square} & Convex & Gaussian measurement & $O(n^2r)$    \\
  \cite{huang2014provable} & Convex & Entry sampling & $O(n^2r\log^2(n))$     \\
  \cite{xia2017polynomial} & Non-Convex &  Entry sampling &  $O(n^{3/2}r^{7/2}\log^{7/2}n+nr^7\log^6n)$ \\
  \cite{yuan2016tensor}& Convex & Entry sampling &$O((n^{3/2}r^{1/2}+nr^2)\log^2(n))$\\
 \hline
\end{tabular}
\caption{Summary of provable low-multilinear-rank tensor recovery approaches. Here the sampling complexity is shown for tensors in $\R^{n\times n\times n}$ with multilinearrank $(r,r,r)$, whose degree of freedoms is $O(nr^2+r^3)$.}
\label{table:tensor_m}
\end{table}

The sampling complexity of these provable approaches is still unsatisfied. They are unnecessarily large in the tensor dimension. Take order-$3$ tensors as examples. The sampling complexity of existing approaches is at least $O(n^{3/2})$ in $n$. As a comparison, the degree of freedom of the low-multilinear-rank tensor is $O(nr^2+r^3)$, which is only linear in $n$. To handle real-world applications where the dimension $n$ is usually huge, it is in demand to develop provable tensor recovery algorithms with sampling complexity linear in $n$, which is the main theme of this paper.

In this paper, we develop an efficient provable polynomial-time approach for low-multilinear-rank tensor recovery with sampling complexity linear in $n$. Our approach reaches the optimal sampling complexity in $n$ (but not in $r$), and we call it a near-optimal sampling complexity. Furthermore, our approach has a very low computational complexity. The success of our proposed approach relies on the following two ingredients. 
\begin{itemize}
\item The computation of our approach is done by the \emph{Riemannian Gradient Descent (RGrad)} algorithm. We reformulate \eqref{eq:tensorrecovery} as a constrained least squares as in the following
\begin{equation}
    \min\limits_{\mathcal{X}\in\mathbb{R}^{n_1\times n_2\times\cdots\times n_d}}\frac{1}{2}\norm{\A\mathcal{X}-\bm{y}}_2^2~~
    \text{s.t.}~~\mbox{mulrank}(\mathcal{X})=\rr.
    \label{eq:recovery0}
\end{equation}
Since the set of all multilinear-rank-$\rr$ tensors forms a Riemannian manifold embedded in $\R^{n_1\times\ldots\times n_d}$ \cite{kressner2014low}, it is natural to apply the RGrad algorithm to solve \eqref{eq:recovery0}. The empirical performance of RGrad algorithm has been studied in \cite{kressner2014low} for tensor recovery without a theoretical guarantee. Following \cite{wei2016guarantees,wei2020guarantees}, we interpret the RGrad algorithm as an iterative hard thresholding (IHT) algorithm \cite{blumensath2009iterative,tanner2013normalized,blanchard2015cgiht,kyrillidis2014matrix,rauhut2017low} with subspace projection. Consequently, RGrad can avoid large scale HOSVD \cite{de2000multilinear} at each iteration, and only HOSVD for small tensors of size $O(r_1\ldots r_d)$ is needed. Therefore, the computational complexity is very low.

\item The theoretical analysis of our approach is done with the help of \emph{tensor restricted isometry property (TRIP)} \cite{rauhut2017low}. It is obvious that \eqref{eq:recovery} is a non-convex optimization and the underlying tensor $\T$ is a global minimizer. A crucial issue is that whether or not the RGrad algorithm converges to the global minimizer $\T$ with a small $m$. Similar to \cite{wei2016guarantees}, we will prove that, if $\A$ satisfies TRIP, the RGrad algorithm with a special initialization converges linearly to $\T$. TRIP was first presented in \cite{rauhut2017low} for the analysis of IHT for tensor recovery, and it is also provided there that $\A$ generated by subgaussian tensors or randomized Fourier transforms satisfies TRIP with almost optimal $m$. As a consequence, our main result reveals that the minimum $m$ for a successful tensor recovery by RGrad algorithm could be as small as $O(nr^2d+r^{d+1})$ in $n$ and $r$. In particular, for an order-3 tensor, the number of samples required is $O(nr^2+r^4)$.
\end{itemize}
 
The rest of the paper is organized as follows. We provide some preliminaries of tensors in Section \ref{sec:preliminaries}. Section \ref{sec:alg} presents Riemannian gradient descent and  related algorithms for tensor recovery. Our main results are summurized in Section \ref{sec:theory} with proofs in Section \ref{sec:proofs}. Numerical experiments and conclusion are in the remaining part. 

\section{Notations and Tensor Preliminaries}\label{sec:preliminaries}
In this section, we introduce some preliminaries of tensors. Throughout this paper, tensors are denoted by capital calligraphic letters, matrices by bold capital letters, and vectors by bold lower case letters. For example, $\mathcal{X}\in\mathbb{R}^{n_1\times n_2\times\cdots\times n_d}$ is a real $d$-th order tensor, $\bm{Y}\in\mathbb{R}^{n_1\times n_2}$ is a real $n_1\times n_2$ matrix, and $\bm{z}\in\R^{n}$ is a real length-$n$ vector. 
For $(j_1,\ldots,j_d)$-th entry of a tensor $\mathcal{X}$, we use either $x_{j_1\ldots j_d}$ or $[\mathcal{X}]_{j_1\ldots j_d}$.  Similar notations hold for matrices and vectors. Linear operators on tensors are denoted by script letters. In particular, $\mathscr{I}$ is the identity operator. The operator norm of a linear operator $\mathscr{B}$ is denoted by $\|\mathscr{B}\|$. Multi-indices are in bold face letters such as $\bm{i}$, and we also use the notation $[d]:=\{1,\ldots,d\}$. For two length-$d$ multi-indices $\bm{i}$ and $\bm{j}$, we use $\bm{i}\preceq\bm{j}$ to indicate that $i_k\leq j_k$ for all $k\in[d]$. 


\subsection{Tensor Operations}
Some basic tensor operations are listed in below.
\begin{itemize}
\item \emph{Tensor inner product and norm.} The inner product of two tensors $\mathcal{X},\mathcal{Z}\in\mathbb{R}^{n_1\times n_2\times\cdots\times n_d}$ is the sum of the products of their entries, i.e.,
$$
\langle\mathcal{X,Z}\rangle = \sum\limits_{j_1=1}^{n_1}\sum\limits_{j_2=1}^{n_2}\cdots\sum\limits_{j_d=1}^{n_d}x_{j_1\cdots j_d}z_{j_1\cdots j_d}.
$$
The Frobenius norm of a tensor $\mathcal{X}$ is $\|\mathcal{X}\|_F=\sqrt{ \langle \mathcal{X},\mathcal{X} \rangle}$.

\item \emph{Tensor matricization.} Tensor matricization \cite{kolda2009tensor} is to transform or flattern a tensor into a matrix. The mode-$i$ matricization of a tensor $\mathcal{X}\in\mathbb{R}^{n_1\times n_2\times\cdots\times n_d}$ is denoted by $\mathcal{X}_{(i)}$. It transforms the tensor $\mathcal{X}$ into a matrix $\mathcal{X}_{(i)}$ of size $n_i\times\prod_{j=1, j\neq i}^{d} n_j$ described below. Let $[\mathcal{X}]_{:,\ldots,:,j,:,\ldots,:}\in\R^{n_1\times\ldots\times n_{i-1}\times n_{i+1}\times\ldots\times n_d}$ be the $j$-th mode-$i$ slice of $\mathcal{X}$, i.e., the sub-tensor by fixing the $i$-th index to be $j$. Then 
$$
\mathcal{X}_{(i)}=\left[\begin{matrix}
\mathrm{vec}([\mathcal{X}]_{:,\ldots,:,1,:,\ldots,:}) & \mathrm{vec}([\mathcal{X}]_{:,\ldots,:,2,:,\ldots,:}) &\ldots &\mathrm{vec}([\mathcal{X}]_{:,\ldots,:,n_i,:,\ldots,:})
\end{matrix}\right]^T,
$$
where $\mathrm{vec}(\cdot)$ stands for the vectorization, and $\cdot^T$ is the transpose of a matrix.

\item \emph{Mode-$i$ tensor multiplication.} The mode-$i$ product of a tensor $\mathcal{X}\in\mathbb{R}^{n_1\times n_2\times\cdots\times n_d}$ with a matrix $\V\in\mathbb{R}^{p\times n_i}$ is denoted by $\mathcal{X}\times_i \V$, and it is of size $n_1\times\cdots\times n_{i-1}\times p\times n_{i+1}\times\cdots\times n_d$ with entries
$$
    [\mathcal{X}\times_i \V]_{j_1\cdots j_{i-1}kj_{i+1}\cdots j_d}=\sum\limits_{l=1}^{n_i}x_{j_1\cdots j_{i-1}l j_{i+1}\cdots j_d}u_{kl}.
$$
Therefore, the mode-$i$ multiplication is obtaining linear combinations of mode-$i$ slices. So the mode-$i$ tensor multiplication can be rewritten into a standard matrix multiplication in mode-$i$ matricizations, i.e.,
\begin{equation}\label{eq:modeimul1}
    \mathcal{Y}=\mathcal{X}\times_i \V~\Longleftrightarrow~\mathcal{Y}_{(i)}=\V\mathcal{X}_{(i)},
\end{equation}
which implies
$$
\mathcal{X}\times_i\bm{V}\times_i\bm{W}=\mathcal{X}\times_i(\bm{W}\bm{V})
$$
Furthermore, it is easy to see that, for $i\neq j$,
$$
\mathcal{X}\times_i\bm{V}\times_j\bm{W}=\mathcal{X}\times_j\bm{W}\times_i\bm{V}.
$$
In other words, mode-$i$ and mode-$j$ multiplications are commutable for $i\neq j$. For this reason, for an index set $\mathbb{I}=\{i_1,\ldots,i_p\}$, we denote
$$
\mathcal{X}\times_{i_1}\bm{V}^{(i_1)}\ldots\times_{i_p}\bm{V}^{(i_p)}
:=\mathcal{X}\times_{i\in\mathbb{I}}\bm{V}^{(i)}.
$$
Furthermore, it has been shown \cite{kolda2009tensor} that, for matrices $\bm{V}^{(j)}$, $j=1,\ldots,d$, and tensors $\mathcal{X}$ and $\mathcal{Y}$ of suitable sizes,
\begin{equation}\label{eq:modeproduct}
    \mathcal{Y}=\mathcal{X}\times_{j\in[d]}\V^{(j)}
    \Longleftrightarrow
     \mathcal{Y}_{(i)}
    =\V^{(i)}\mathcal{X}_{(i)}\left(\V^{(d)}\otimes \cdots\otimes \V^{(i+1)}\otimes \V^{(i-1)}\cdots\otimes \V^{(1)} \right)^T 
\end{equation}
Here $\otimes$ means the Kronecker product of matrices.  
\end{itemize}

\subsection{Multilinear Rank and Tensor Manifold}
A tensor $\mathcal{X}\in\R^{n_1\times\ldots\times n_d}$ is said of multilinear rank $\bm{r}=(r_1,\ldots,r_d)$ if $\mathrm{rank}(\mathcal{X}_{(i)})=r_i$ for $i=1,\ldots,d$ \cite{kolda2009tensor}. In other words,
$$
\mbox{mulrank}(\mathcal{X})=\bm{r}
\quad\Longleftrightarrow\quad
\mathrm{rank}(\mathcal{X}_{(i)})=r_i,~~i=1,\ldots,d.
$$
Let $\bm{V}^{(i)}\in\R^{n_i\times r_i}$ be an orthonomal basis of the column space of $\mathcal{X}_{(i)}$ for $i=1,\ldots,d$ respectively. Then, $\mathcal{X}$ can be rewritten into a Tucker decomposition form
\begin{equation}\label{eq:Tucker2}
\mathcal{X}=\mathcal{B}\times_{i\in[d]}\V^{(i)},
\end{equation}
where $\mathcal{B}\in\R^{r_1\times\ldots\times r_d}$ has a full multilinear rank and is called the core tensor. This factorization \eqref{eq:Tucker2} is exactly the same as \eqref{eq:Tucker1}.

The collection of tensors of multilinear rank $\bm{r}$ forms a smooth embedded submanifold of $\mathbb{R}^{n_1\times\cdots\times n_d}$, denoted by $\mathbb{M}_{\bm{r}}$, i.e.,
$$
\mathbb{M}_{\bm{r}}=\{\mathcal{X}\in\R^{n_1\times\ldots\times n_d}~|~\mbox{mulrank}(\mathcal{X})=\bm{r}\}.
$$
The structure of this manifold is studied in \cite{kressner2014low,koch2010dynamical}. Let $\mathcal{X}\in\mathbb{M}_{\bm{r}}$ be with Tucker factorization \eqref{eq:Tucker2}. By differentiating the factors of \eqref{eq:Tucker2},  we obtain the tangent space of $\mathbb{M}_{\bm{r}}$ at $\mathcal{X}$ as
\begin{equation}\label{eq:tangent}
\begin{split}
\S_{\mathcal{X}}=\left\{\dot{\mathcal{B}}\times_{i\in[d]}\V^{(i)}
+\sum_{i=1}^d\mathcal{B}\times_{j\in[d]\setminus i}\V^{(j)}\times_{i}\dot{\bm{V}}^{(i)}\Big|~\dot{\mathcal{B}}\in\mathbb{R}^{r_1,\times\ldots\times r_d},\quad\dot{\V}^{(i)}\in\mathbb{R}^{n_i\times r_i},~(\dot{\V}^{(i)})^T\V^{(i)}=\bm{0},~i=1,\ldots,d.\right\}
\end{split}
\end{equation}

We can see that the $d+1$ components in the first line of \eqref{eq:tangent} are orthogonal to each others. Therefore, the dimension of the tangent space (hence the manifold) is total number of free parameters in \eqref{eq:tangent}, i.e.,
$$
\dim(\mathbb{M}_{\bm{r}})=\prod_{i=1}^d r_i+\sum_{j=1}^d(r_jn_j-r_j^2).
$$
Thus, the degree of freedoms in a multilinear-rank-$\bm{r}$ tensor is $O(d(nr-r^2)+r^d)$ if $n=\max\{n_1,\ldots,n_d\}$ and $r=\max\{r_1,\ldots,r_d\}$. 

By re-orthogonalization of $[\V^{(i)}~\dot{\V}^{(i)}]\in\mathbb{R}^{n_i\times 2r_i}$ to obtain an orthogonal matrix $\bm{W}^{(i)}\in\mathbb{R}^{n_i\times 2r_i}$ for $i=1,\ldots,d$, we see that any tensor $\mathcal{Y}\in\S_{\mathcal{X}}$ can be rewritten into a multilinear factorization form as $\mathcal{Y}=\widetilde{\mathcal{B}}\times_{i\in[d]}\bm{W}^{(i)}$, where $\widetilde{\mathcal{B}}\in\mathbb{R}^{2r_1\times\ldots\times2r_d}$ is a core tensor containing $\mathcal{B}$, $\dot{\mathcal{B}}$,  and the re-orthogonalization coefficients. Therefore, tensors in tangent spaces of $\mathbb{M}_{\bm{r}}$ always have a multilinear tensor at most $2\bm{r}$.

\section{Algorithms}\label{sec:alg}
We reformulate the low-mulitilinear-rank tensor recovery problem \eqref{eq:tensorrecovery} as the constrained least squares \eqref{eq:recovery0}, i.e.,
\begin{equation}
    \min\limits_{\mathcal{X}\in\mathbb{M}_{\bm{r}}}\frac{1}{2}\norm{\A\mathcal{X}-\bm{y}}_2^2.
    \label{eq:recovery}
\end{equation}
Obviously, the underlying unknown tensor $\mathcal{X}$ is a global minimizer of \eqref{eq:recovery} as long as $\A$ is injective on $\mathbb{M}_{\bm{r}}$. Therefore, solving the low-multilinear-rank tensor recovery problem \eqref{eq:tensorrecovery} is equivalent to finding a global minimizer of \eqref{eq:recovery}. In this section, we present algorithms for \eqref{eq:recovery}. In particular, an iterative hard thresholding (IHT) algorithm and a Riemannian gradient descent (RGrad) algorithm will be introduced.
 
\subsection{Iterative Hard Thresholding (IHT)}
A natural solver for constrained optimization is the projected gradient descent, which applies to \eqref{eq:recovery} to yield
\begin{equation}\label{eq:PGD}
\mathcal{T}_{l+1}=\P_{\mathbb{M}_{\bm{r}}}\left(\mathcal{T}_l-\alpha_l\A^{*}(\A\mathcal{T}_l-\bm{y})\right),
\end{equation}
where $\P_{\mathbb{M}_{\bm{r}}}$ is the projection onto the multilinear-rank-$\bm{r}$ tensor manifold $\mathbb{M}_{\bm{r}}$, and $\alpha_l$ is a step size. If $\mathbb{M}_{\bm{r}}$ is replaced by the set of sparse vectors and the set of low-rank matrices respectively, the algorithm is widely used in compress sensing and low rank matrix recovery, known as iterative hard thresholding (IHT) \cite{blumensath2009iterative,tanner2013normalized,wei2016guarantees, goldfarb2011convergence}. 

The iteration \eqref{eq:PGD} does not become a practical algorithm unless $\P_{\mathbb{M}_{\bm{r}}}$ can be obtained easily. Unfortunately, it is quite difficult to have an efficient algorithm for computing the exact $\P_{\mathbb{M}_{\bm{r}}}(\mathcal{Y})$ for a given tensor $\mathcal{Y}$. Instead we then compute an approximation of  $\P_{\mathbb{M}_{\bm{r}}}(\mathcal{Y})$. There are several possible strategies. Here we use the one based on the truncated Higher-order SVD (HOSVD) \cite{tucker1966some,de2000multilinear}, presented in Algorithm \ref{alg:HOSVD}.

\begin{algorithm}[h]
\KwIn{A tensor $\mathcal{Y}\in\mathbb{R}^{n_1\times\cdots\times n_d}$ and the target rank $\bm{r}=(r_1, r_2, \dots, r_d)$}
\For{$i=1,\cdots, d$}{
$\V^{(i)}\gets r_i$ dominant left singular vectors of $\mathcal{Y}_{(i)}$
}
Compute the truncated core tensor: $\mathcal{B}=\mathcal{Y}\times_{i\in[d]}(\V^{(i)})^T$\\
\KwResult{$\mathscr{H}_{\bm{r}}(\mathcal{Y})=\mathcal{B}\times_{i\in[d]}\V^{(i)}$}
\caption{Truncated Higher-order SVD (HOSVD)~\cite{tucker1966some,de2000multilinear}}
\label{alg:HOSVD}
\end{algorithm}

In the truncated HOSVD $\mathscr{H}_{\bm{r}}(\mathcal{Y})$ of Algorithm \ref{alg:HOSVD},  we compute the $r_i$ principal components $\V^{(i)}=[\bm{v}^{(i)}_{1}~\ldots~\bm{v}^{(i)}_{r_i}]\in\mathbb{R}^{n_i\times r_i}$ of mode-$i$ matricizations of $\mathcal{Y}$ for all $i$, and then $\mathcal{Y}$ is projected into the span of the orthogonal set 
$$
\left\{\bm{v}^{(1)}_{j_1}\otimes\ldots\otimes\bm{v}^{(d)}_{j_d}~|~j_i\in[r_i],~i\in[d].\right\}.
$$
The computational cost of Algorithm \ref{alg:HOSVD} is $d$ matrix SVDs of size $n_i\times\left(\prod_{j\neq i}n_j\right)$, $i=1,\ldots,d$, respectively. It can be shown that the truncated HOSVD $\mathscr{H}_{\bm{r}}$ is indeed a good approximation of $\P_{\mathbb{M}_{\bm{r}}}$ in the sense that
\begin{equation}\label{eq:quasibest}
\|\mathcal{Y}-\mathscr{H}_{\bm r}(\mathcal{Y})\|_F\le \sqrt{d} \|\mathcal{Y}-\P_{\mathbb{M}_{\bm{r}}}(\mathcal{Y})\|_F,
\qquad
\forall~\mathcal{Y}\in\R^{n_1\times\ldots\times n_r}.
\end{equation}
We call $\mathscr{H}_{\bm r}$ a quasi-projection, as it produces a quasi-optimal low-multilinear-rank approximation. Some other strategies are also available. For example, one can use a successive version of truncated HOSVD, where columns of $\V^{(i)}$ are the $r_i$ leading singular vectors of $\left(\mathcal{Y}\times_{j\in[i-1]}(\V^{(j)})^T\right)_{(i)}$; and the computational cost is lower than the standard truncated HOSVD while still achieving a quasi-optimal low-multilinear-rank approximation as in \eqref{eq:quasibest}.

When the exact projection $\P_{\mathbb{M}_{\bm{r}}}$ is replaced by the quasi-projection $\mathscr{H}_{\bm r}$, we obtain a practical algorithm
\begin{equation}\label{eq:IHT}
\mathcal{T}_{l+1}=\mathscr{H}_{\bm{r}}\left(\mathcal{T}_l-\alpha_l\A^{*}(\A\mathcal{T}_l-\bm{y})\right).
\end{equation}
This algorithm is developed in \cite{rauhut2017low} as an Iterative Hard Thresholding (IHT) algorithm for low-rank Tensor recovery. Since the objective function in \eqref{eq:recovery} is quadratic, an exact line search step size is available. Altogether, we obtain the tensor Normalized IHT (NIHT) algorithm \cite{rauhut2017low} shown in Algorithm \ref{alg:niht}. In the formula for step size, we used the linear projection operator $\mathscr{J}_l$. In particular, let 
$$
\T_l=\mathcal{C}_l\times_{i\in[d]}\U_l^{(i)}
$$
be the multilinear factorization of $\T_l$, where $\mathcal{C}_l\in\R^{r_1\times\ldots\times r_d}$ is the core tensor of $\T_l$, and $\U_l^{(i)}$, $i=1,\ldots,d$, are the orthogonal factors of $\T_l$. Then $\mathscr{J}_l$ is the projection into the subspace of tensors whose mode-$i$ fibres are spanned by $\U_l^{(i)}$, i.e.,
$$
\mathscr{J}_l\mathcal{G}_l
=\mathcal{G}_l\times_{i\in[d]}(\U_l^{(i)} (\U_l^{(i)})^T)
$$
 
\begin{algorithm}[h]
\textbf{Initialization}: $\T_0$ (usually $\T_0=0$ )\\
\For{$l=0,1,\cdots$}{
$\mathcal{G}_l =\A^{*}(y-\A(\mathcal{T}_l))$\\
$\mu_l=\frac{\|\mathscr{J}_l\mathcal{G}_l\|_F^2}{\|\A\mathscr{J}_l\mathcal{G}_l\|_2^2}$\\
$\mathcal{V}_l=\mathcal{T}_l+\mu_l\mathcal{G}_l$\\
$\mathcal{T}_{l+1}=\mathscr{H}_{\bm{r}}(\mathcal{V}_l)$
}
\caption{Tensor Normalized iterative hard thresholding (NIHT)~\cite{rauhut2017low}}
\label{alg:niht}
\end{algorithm}

Under the tensor restricted isometric property (TRIP) and some other assumptions, NIHT is shown to be convergent \cite{rauhut2017low} to the global minimizer $\mathcal{T}$. However, the verification of a key inequality in the proof is missing in \cite{rauhut2017low}. Therefore, the recovery guarantee of IHT type algorithms is still uncertain. 

\subsection{IHT with Subspace Projection and Riemannian Gradient Descent (RGrad)}\label{sec:algrgrad}
The tensor IHT algorithm \eqref{eq:IHT} suffers from the following two issues. Firstly, the computational cost of IHT is high, because the truncated HOSVD of an $n_1\times \ldots \times n_d$ tensor needs to be computed per iteration. Secondly, the recovery guarantee of IHT is still incomplete and the sampling complexity is still unknown. To overcome these drawbacks, we use the Riemannian gradient descent (RGrad) algorithm to solve \eqref{eq:recovery}, as the set $\mathbb{M}_{\bm{r}}$ of all multilinear-rank-$\bm{r}$ tensors forms a smooth embedded manifold in $\R^{n_1\times\ldots\times n_d}$. We will see later that the RGrad algorithm not only achieves a low computational complexity by avoiding large scale truncated HOSVD, but also has a theoretical guarantee with a sampling complexity $m\sim O(dnr^2+r^{d+1})$.

For simplicity and following \cite{wei2016guarantees, wei2020guarantees}, we present the GRrad algorithm algebraically rather than geometrically. More precisely, we interpret the GRrad algorithm as an IHT algorithm with subspace projection. To improve the computational efficiency of IHT algorithms, our idea is to relief the computation of $\mathscr{H}_{\bm{r}}$, the most computational consuming step in the IHT algorithm \eqref{eq:IHT}. Let $\T_l$ be the estimation of $\T$ at the $l$-th iteration. Since $\mathscr{H}_{\bm{r}}$ is a quasi-projection onto $\mathbb{M}_{\bm{r}}$, which can be approximated well by its tangent space $\S_l$ at $\T_l$, we use the projection $\mathscr{P}_{\S_l}$ onto $\S_l$ to approximate $\mathscr{H}_{\bm{r}}$ in \eqref{eq:IHT}.  
However, the multilinear rank of tensors in $\S_l$ is $2\bm{r}$ rather than $\bm{r}$. Therefore, we then use the truncated HOSVD $\mathscr{H}_{\bm{r}}$ to trim the multilinear rank back to $\bm{r}$. Altogether, the IHT with subspace projection becomes
\begin{equation}\label{eq:IHTWSP}
\mathcal{T}_{l+1}=\mathscr{H}_{\bm{r}}\mathscr{P}_{\S_l}\left(\mathcal{T}_l-\alpha_l\A^{*}(\A\mathcal{T}_l-\bm{y})\right).
\end{equation}
Since $\T_l\in\S_l$, the algorithm \eqref{eq:IHTWSP} can be rewritten as
\begin{equation}\label{eq:RGrad}
\mathcal{T}_{l+1}=\mathscr{H}_{\bm{r}}\left(\mathcal{T}_l-\alpha_l\mathscr{P}_{\S_l}\A^{*}(\A\mathcal{T}_l-\bm{y})\right).
\end{equation}
This is exactly the Riemannian gradient descent (RGrad) algorithm for solving \eqref{eq:recovery}, since $\mathscr{P}_{\S_l}\A^{*}(\A\mathcal{T}_l-\bm{y})$ is exactly the Riemannian manifold gradient of the objective function in \eqref{eq:recovery} on $\mathbb{M}_{\bm{r}}$. The operator $\mathscr{H}_{\bm{r}}$ in \eqref{eq:RGrad} serves as a retraction operator. The RGrad algorithm is summarized in Algorithm \ref{alg:Rgrad}. RGrad algorithms have been studied empirically for low-rank tensor recovery in \cite{kressner2014low}. However, it is still unknown whether or not RGrad converges to the global minimizer, and the sampling complexity is not clear.

\begin{algorithm}[h]
\textbf{Initialization}: $\T_0=\mathscr{H}_{\rr}(\A^{*}\bm{y})$\\
\For{$l=0,1,\cdots$}{
$\mathcal{G}_l = \A^*(\A\T_l-\bm{y})$\\
$\alpha_l = \frac{\|\P_{\S_l}\mathcal{G}_l\|_F^2}{\|\A\P_{\S_l}\mathcal{G}_l\|_2^2}$\\
$\mathcal{W}_l=\T_l-\alpha_l\P_{\S_l}\mathcal{G}_l$\\
$\T_{l+1}=\mathscr{H}_r(\mathcal{W}_l)$
}
\caption{Riemannian Gradient Descent (IHT with Subspace Projection)}
\label{alg:Rgrad}
\end{algorithm}

At the first glance, since we added one more operator $\mathscr{P}_{\S_l}$ into \eqref{eq:IHT}, the RGrad algorithm \eqref{eq:IHTWSP} (or \eqref{eq:RGrad}) seems more computational demanding than the IHT algorithm \eqref{eq:IHT}. However, this is not the truth. Actually, because tensors in $\S_l$ has a multilinear-rank at most $2\bm{r}$, $\mathscr{H}_{\bm{r}}\mathscr{P}_{\S_l}$ together in \eqref{eq:IHTWSP} is much easier to compute than $\mathscr{H}_{\bm{r}}$ solely in \eqref{eq:IHT}. This can be seen from the detailed implementation of \eqref{eq:IHTWSP}, which is divided into the following two paragraphs.

\paragraph{Computation of $\mathscr{P}_{\S_l}$.}
As shown in Algorithm \ref{alg:Rgrad}, denote $\mathcal{G}_l=\A^*(\A\T_l-\bm{y})$. We need to compute $\mathscr{P}_{\S_l}\mathcal{G}_l$. Let the HOSVD of $\T_l$ be
$$
\T_l=\mathcal{C}_l\times_{i\in[d]}\bm{U}_l^{(i)},
$$ 
where $\mathcal{C}_l\in\R^{r_1\times\ldots\times r_d}$ is the core tensor and $\bm{U}_l^{(i)}\in\R^{n_i\times r_i}$ is the $i$-th orthogonal factor for $i=1,\ldots,d$. According to \eqref{eq:tangent}, $\mathscr{P}_{\S_l}\mathcal{G}_l$ must be in the form of
\begin{equation}\label{eq:projtan}
\mathscr{P}_{\S_l}\mathcal{G}_l=\mathcal{D}_l\times_{i\in[d]}\U_l^{(i)}
+\sum_{i=1}^d\mathcal{C}_l\times_{j\in[d]\setminus i}\U_l^{(j)}\times_{i}\bm{W}_l^{(i)},
\end{equation}
where $\mathcal{D}_l\in\R^{r_1\times\ldots\times r_d}$ is arbitrary, and $\bm{W}_l^{(i)}\in\R^{n_i\times r_i}$ satisfies $(\bm{W}_l^{(i)})^T\U_l^{(i)}=\bm{0}$ for $i=1,\ldots,d$. The summands in the right hand side of \eqref{eq:projtan} are orthogonal to each other, so that they can be obtained independently by solving the least squares 
$$
\mathcal{D}_l=\underset{\mathcal{D}\in\mathbb{R}^{r_1,\times\ldots\times r_d}}{\mathrm{argmin}}\left\|\mathcal{G}_l-\mathcal{D}\times_{i\in[d]}\U_l^{(i)}\right\|_F^2
$$
and
\begin{equation}\label{eq:LSinprojtan}
\bm{W}_l^{(i)}=\underset{\substack{\bm{W}\in\mathbb{R}^{n_i\times r_i}\\ \bm{W}^T\U_l^{(i)}=\bm{0}}}{\mathrm{argmin}}\left\|\mathcal{G}_l-\mathcal{C}_l\times_{j\in[d]\setminus i}\U_l^{(j)}\times_{i}\bm{W}\right\|_F^2
\end{equation}
for $i=1,\ldots,d$, respectively. The closed form solutions are given as follows.
\begin{itemize}
\item For $\mathcal{D}_l$: Let $\bm{u}_{l,j}^{(i)}$ be the $j$-th column of $\U_l^{(i)}$. Since $\left\{\bm{u}^{(1)}_{l,j_1}\otimes\ldots\otimes\bm{u}^{(d)}_{l,j_d}~|~j_i\in[n_i],~i\in[d]\right\}$ is an orthonormal set, we have $[\mathcal{D}_l]_{j_1\ldots j_d}=\langle\mathcal{G}_l,\bm{u}^{(1)}_{l,j_1}\otimes\ldots\otimes\bm{u}^{(d)}_{l,j_d}\rangle$, or equivalently,
\begin{equation}\label{eq:dBinprojtan}
\mathcal{D}_l=\mathcal{G}_l\times_{i\in[d]}(\U_l^{(i)})^T.
\end{equation}
This is can be done efficiently by matrix-vector products. The computational complexity is at most $O\left(\sum_{j=1}^d\left(\prod_{k=1}^{j}r_k\right)\left(\prod_{k=j}^{d}n_k\right)\right)\leq O(dn^dr)$.

\item For $\bm{W}_l^{(i)}$: Thanks to \eqref{eq:modeproduct}, we see that \eqref{eq:LSinprojtan} is equivalent to the following constrained matrix least squares problem
$$
\bm{W}_l^{(i)}=\underset{\substack{\bm{W}\in\mathbb{R}^{n_i\times r_i}\\ \bm{W}^T\U_l^{(i)}=\bm{0}}}{\mathrm{argmin}}\left\|(\mathcal{G}_l)_{(i)}-\bm{W}(\mathcal{C}_l)_{(i)}\left(\U_l^{(d)}\otimes \cdots\otimes \U_l^{(i+1)}\otimes \U_l^{(i-1)}\cdots\otimes \U_l^{(1)} \right)^T\right\|_F^2.
$$
It is not difficult to see that the closed form solution is
\begin{equation}\label{eq:dVinprojtan}
\begin{split}
\bm{W}_l^{(i)}&=\left(\bm{I}-\U_l^{(i)}(\U_l^{(i)})^T\right)(\mathcal{G}_l)_{(i)}\left(\U_l^{(d)}\otimes \cdots\otimes \U_l^{(i+1)}\otimes \U_l^{(i-1)}\otimes\cdots\otimes \U_l^{(1)} \right)(\mathcal{C}_l)_{(i)}^{\dag}\cr
&=\left(\mathcal{G}_l\times_{j\in[d]\setminus i} (\U_l^{(j)})^T\times_i(\bm{I}-\U_l^{(i)}(\U_l^{(i)})^T)
\right)_{(i)}(\mathcal{C}_l)_{(i)}^{\dag},
\end{split}
\end{equation}
where $(\mathcal{C}_l)_{(i)}^{\dag}$ is the pseudo-inverse of $(\mathcal{C}_l)_{(i)}$. This is again done by matrix-vector products and a small matrix inversion. The computational complexity is at most $O(dn^dr+nr^d+r^{d+1})=O(dn^dr)$.
\end{itemize}

\paragraph{Computation of $\mathscr{H}_{\bm{r}}$.}
The same as in Algorithm \ref{alg:Rgrad}, we define $\mathcal{W}_l=\T_l-\alpha_l\mathscr{P}_{\S_l}\mathcal{G}_l$. To obtain the next estimate $\T_{l+1}$, we retract the tensor $\mathcal{W}_l$ from the tangent space $\S_l$ to the manifold $\mathbb{M}_{\rr}$ by $\T_{l+1}=\mathscr{H}_{\bm{r}}(\mathcal{W}_l)$. To implement the retraction efficiently, we use the fact that all tensors on tangent spaces of $\mathbb{M}_{\bm{r}}$ has a multilinear-rank at most $2\bm{r}$. By direct calculation,
\begin{equation}\label{eq:retraction}
\begin{split}
\mathcal{W}_l&=\T_l-\alpha_l\P_{\S_l}(G_l)=\mathcal{C}_l\times_{i\in[d]}\U_l^{(i)}-
\alpha_l\left(\mathcal{D}_l\times_{i\in[d]}\U_l^{(i)}
+\sum_{i=1}^d\mathcal{C}_l\times_{j\in[d]\setminus i}\U_l^{(j)}\times_{i}\bm{W}_l^{(i)}\right)\cr
&=(\mathcal{C}_l-\alpha_l\mathcal{D}_l)\times_{i\in[d]}\U_l^{(i)}-\alpha_l\sum_{i=1}^d\mathcal{C}_l\times_{j\in[d]\setminus i}\U_l^{(j)}\times_{i}\bm{W}_l^{(i)}:=\mathcal{L}_l\times_{i\in[d]}\left[\U_l^{(i)}~\bm{W}_l^{(i)}\right],
\end{split}
\end{equation}
where $\mathcal{L}_l\in\mathbb{R}^{2r_1\times\ldots\times 2r_d}$ is a block tensor whose $(1:r_1,\ldots,1:r_d)$ sub-tensor is $\mathcal{C}_l-\alpha_l\mathcal{D}_l$ and $(1:r_1,\ldots,1:r_{i-1},r_{i}+1:2r_i,1:r_{i+1},\ldots,1:r_d)$ sub-tensor is $-\alpha_l\bm{C}_l$ for $i=1,\ldots,d$. We can exploit this structure to compute the truncated HOSVD of $\mathcal{W}_l$ efficiently. We first compute the QR decomposition 
\begin{equation}\label{eq:QR}
\left[\U_l^{(i)}~\bm{W}_l^{(i)}\right]
=\bm{Q}_l^{(i)}\bm{R}_l^{(i)},
\qquad i=1,\ldots,d,
\end{equation}
which leads to a re-expression of $\mathcal{W}_l$ as
\begin{equation}\label{eq:Ltilde}
\begin{split}
\mathcal{W}_l=\mathcal{L}_l\times_{i\in[d]}(\bm{Q}_l^{(i)}\bm{R}_l^{(i)})
=\left(\mathcal{L}_l\times_{i\in[d]}\bm{R}_l^{(i)}\right)\times_{i\in[d]}\bm{Q}_l^{(i)}
:=\widetilde{\mathcal{L}}_l\times_{i\in[d]}\bm{Q}_l^{(i)}.
\end{split}
\end{equation}
Because $\bm{Q}_l^{(i)}$ are orthonormal for $i=1,\ldots,d$, it suffices to compute the HOSVD of the small tensor $\widetilde{\mathcal{L}}_l\in\R^{2r_1\times\ldots\times 2r_d}$, instead of the large tensor $\mathcal{W}_l$ directly. To this end, we compute the SVD of matricizations of $\widetilde{\mathcal{L}}_l\in\R^{2r_1\times\ldots\times 2r_d}$ to obtain 
\begin{equation}\label{eq:tildeU}
\widetilde{\U}_{l}^{(i)}:=\mbox{Left singular vectors of }(\widetilde{\mathcal{L}}_{l})_{(i)},\qquad i=1,\ldots,d.
\end{equation}
For $i=1,\ldots,d$, since $\bm{Q}_l^{(i)}$ are orthonormal, $\bm{Q}_l^{(i)}\widetilde{\U}_{l}^{(i)}\in\R^{n_i\times 2r_i}$ are the left singular vectors of $\mathcal{W}_{l}$. So, we obtain the HOSVD of $\mathcal{W}_l$. To obtain the truncated HOSVD of $\mathcal{W}_l$, we let 
\begin{equation}\label{eq:Ul+1}
\U_{l+1}^{(i)}:=\bm{Q}_l^{(i)}\left[\widetilde{\U}_{l}^{(i)}\right]_{:,1:r_i}\in\R^{n_i\times r_i},\qquad i=1,\ldots,d,
\end{equation}
and
\begin{equation}\label{eq:Cl+1}
\mathcal{C}_{l+1}:=\mathcal{W}_l\times_{i\in[d]}(\U_{l+1}^{(i)})^T=\widetilde{\mathcal{L}}_l\times_{i\in[d]}\left[\widetilde{\U}_{l}^{(i)}\right]_{:,1:r_1}^T
\end{equation}
So, $\T_{l+1}$ is expressed implicitly in HOSVD form as in the following
$$
\T_{l+1}=\mathscr{H}_{\bm{r}}(\mathcal{W}_l)
=\mathcal{C}_{l+1}\times_{i\in[d]}\U_{l+1}^{(i)}.
$$ 
We see that the main computation is $d$ QR decompositions of sizes $n_i\times 2r_i$ in \eqref{eq:QR}, $d$ tensor multiplications  to form $\widetilde{\mathcal{L}}_l$ in \eqref{eq:Ltilde}, $d$ matrix SVD's of sizes $2r_i\times\prod_{j\neq i}2r_j$ in \eqref{eq:tildeU}, $d$ matrix-matrix products of sizes $n_i\times 2r_i$ and $2r_i\times r_i$ for $\U_{l+1}^{(i)}$ in \eqref{eq:Ul+1}, and $d$ tensor multiplications to form $\mathcal{C}_{l+1}$ in \eqref{eq:Cl+1}, for $i=1,\ldots,d$ respectively. Ignoring the lower-order terms, the total computational complexity is $O(dn^dr+(2r)^{d+1})$.

\bigskip

Therefore, with the above implementation, the total computational complexity of Algorithm \ref{alg:Rgrad} is $O(d^2n^dr+(2r)^{d+1})$ plus that for two applications of $\A$ and one application of $\A^*$ in the computation of the gradient and the step size. For general $\A$, it costs $O(mn^d)$ to compute the application of $\A$ and $\A^*$ respectively. So, the total computational complexity of Algorithm \ref{alg:Rgrad} is $O(d^2n^dr+mn^d+(2r)^{d+1})$. For some special $\A$, the computational cost can be less. For example, in the tensor completion case where $\A$ samples entries, the application of $\A$ and $\A^*$ costs only $O(m)$ and the gradient $\mathcal{G}_l$ is a sparse tensor. By this way, the computational cost can be significantly reduced.

\section{Recovery Guanrantee}\label{sec:theory}
In this section, we give the recovery guarantee of RGrad algorithm. In particular, we prove that Algorithm \ref{alg:Rgrad} converges linearly to the underlying tensor $\T$, provided that the sampling operator $\A$ satisfies the so-called tensor restricted isometry property (TRIP) \cite{rauhut2017low}. Our results reveals that, if the measurement tensors $\mathcal{A}_i$, $i=1,\ldots,m$ in $\A$ are subgaussian random with i.i.d. entries, then Algorithm \ref{alg:Rgrad} is able to recover $\T$ exactly with $m~O(nr^2+r^{d+1})$. Thus, we obtain a provable tensor recovery algorithm with a sampling complexity optimal in $n$.

\subsection{Tensor Restricted Isometry Property (TRIP)}
Our main result is proved under the assumption that $\A$ satisfies the tensor restricted isometry property (TRIP) proposed in \cite{rauhut2017low}. TRIP is a generalization of matrix RIP for low-rank matrix recovery \cite{recht2010guaranteed}. Under the matrix RIP, many algorithms are guaranteed to have a successful low-rank matrix recovery \cite{recht2010guaranteed,wei2016guarantees}. Since there are several definitions of tensor ranks, TRIP can be defined for different tensor ranks. Here we present only TRIP for the multilinear rank. 

\begin{defn}[{\cite{rauhut2017low}}]
Let $\A~:~\R^{n_1\times\ldots\times n_d}\to\R^m$ be a linear operator. Let $\bm{s}$ be a rank tuple. We say that $\A$ satisfies the tensor restricted isometry property (TRIP) if there exists a constant $\delta_{\bm{s}}\in(0,1)$ such that 
\begin{align*}
    (1-\delta_{\bm{s}})\|\mathcal{X}\|_{F}^{2}\leq \|\A\mathcal{X}\|_2^2\leq(1+\delta_{\bm{s}})\|\mathcal{X}\|_F^2
\end{align*}
holds for all tensors $\mathcal{X}$ satisfying $\mathrm{mulrank}(\mathcal{X})\preceq\bm{s}$. The constant $\delta_{\bm{s}}$ is called the restricted isometry constant (RIC).
\end{defn}

It is shown in \cite{rauhut2017low} that there are several types of $\A$ satisfying TRIP. One typical such an operator is the subgaussian random operator $\A$, whose measurement tensors $\mathcal{A}_i$ for $i=1,\ldots,m$ are generated by i.i.d. subgaussian random entries. A subgaussian random operator $\A$ satisfies TRIP with multilinear rank $\bm{s}$ and constant $\delta_{\bm{s}}$ with probability exceeding $1-\epsilon$ provided that
\begin{equation}\label{eq:msubgau}
    m\geq C \delta_{\bm{s}}^{-2} \max\{ (s^d+dns)\log(d),\log(\epsilon^{-1})\}
\end{equation}
where $C$ is the constant that depends on the subgaussian parameter, and $s=\|\bm{s}\|_{\infty}$. Another typical operator satisfying TRIP is random Fourier mapping, for which we omit the details.

\subsection{Main Results and Sampling Complexity}
The main theoretical result of this paper is that Algorithm \ref{alg:Rgrad} converges linearly to the underlying true tensor $\T$ provided $\A$ satisfies TRIP. The result is summarized into the following theorem.

\begin{theorem}[Recovery Guarantee of Riemannian Gradient Descent]\label{Rguarantee}
Let $\A: \mathbb{R}^{n_1\times n_2\times\cdots\times n_d}\to\mathbb{R}^m$ with $n=\max\{n_1,\ldots,n_d\}$ be a linear map. Let $\T\in\R^{n_1\times\ldots\times n_d}$ be a tensor of multilinear rank $\bm{r}$ with $r=\|\bm{r}\|_{\infty}$. Let $\bm{y}=\A\T$. Assume $\A$ satisfies TRIP with constants $\delta_{2\bm{r}}$. Define
\begin{equation}\label{eq:gamma}
\gamma=\frac{2\delta_{2\bm{r}}}{1-\delta_{2\bm{r}}}
(\sqrt{d}+1)\left(1+(2^d-1)(\sqrt{d}+1)\frac{\norm{\T}_F}{\min\limits_{i}(\sigma_{r_i}(\T_{(i)}))}\right).
\end{equation}
Then provided $\gamma<1$, the sequence $\{\T_l\}_{l\in\mathbb{N}}$ generated by Algorithm \ref{alg:Rgrad} satisfies 
$$
\norm{\T_{l}-\T}_F \leq \gamma^l\norm{\T_0-\T}_F.
$$
In particular, $\gamma<1$ can be satisfied if 
\begin{equation}\label{eq:ric}
\delta_{2\rr}\leq \delta:=\frac{1}{3\cdot2^d(\sqrt{d}+1)^2\kappa\sqrt{r}},
\end{equation}
where $\kappa$ is a condition number  of $\T$ defined by $\kappa=\frac{\min_i\sigma_{1}(\T_{(i)})}{\min_i\sigma_{r_i}(\T_{(i)})}$.
\end{theorem}

The proof of Theorem \ref{Rguarantee} is postponed to the next section. Combining Theorem \ref{Rguarantee} and the result in \cite{rauhut2017low}, we obtain the following corollary on the sampling complexity of Algorithm \ref{alg:Rgrad}.

\begin{corollary}[Sampling Complexity of Riemannian Gradient Descent]\label{cor:SamComplexity}
Let $\A: \mathbb{R}^{n_1\times n_2\times\cdots\times n_d}\to\mathbb{R}^m$ be a linear map generated by \eqref{eq:meaA} with entries of $\mathcal{A}_i$ for $i=1,\ldots,m$ drawn from i.i.d. mean-$0$ variance-$\frac{1}{m}$ subgaussian distributions. Let $\T$, $\bm{y}$, $\kappa$, $n$, $\bm{r}$, and $r$ are the same as in Theorem \ref{Rguarantee}. Then, with probability at least $1-e^{-dnr}$, the sequence $\{\T_l\}_{l\in\mathbb{N}}$ generated by Algorithm \ref{alg:Rgrad} satisfies 
$$
\norm{\T_{l}-\T}_F \leq \gamma^l\norm{\T_0-\T}_F
$$
for some constant $\gamma\in(0,1)$, provided
\begin{equation}\label{eq:mbound}
m\geq CC_d\kappa^2(dnr^2+r^{d+1}),
\end{equation}
where $C$ is a constant depending only on the subgaussian parameter, and $C_d=d^24^{d}\log d$ is a constant depending only on $d$.
\end{corollary}
\begin{proof}
We choose parameters in \eqref{eq:msubgau} as $\bm{s}=2\bm{r}$, $\epsilon=e^{-dnr}$, and $\delta_{\bm{s}}=\delta$ as defined in \eqref{eq:ric}. A simplification leads to the bound \eqref{eq:mbound} of $m$.
\end{proof}

Therefore, if $d$ is a constant, then we need only $O(dnr^2+r^{d+1})$ subgaussian samples to use Algorithm \ref{alg:Rgrad} to recovery an order-$d$ tensor of size $n_1\times\ldots\times n_d$ and multilinear rank $\bm{r}$. As a comparison, the degree of freedoms in the tensor is $O(dnr+r^d)$. So the sampling complexity of Algorithm \ref{alg:Rgrad} is optimal in $n$. In particular, for a $3$-rd tensor, the sampling complexity of our method is $O(nr^2+r^4)$ that is linear in $n$, while other existing provable approaches \cite{mu2014square,huang2014provable,xia2017polynomial,yuan2016tensor} are either $O(n^{1.5})$ or $O(n^2)$ in $n$.

\section{Proofs}\label{sec:proofs}
In this section, we prove our main result Theorem \ref{Rguarantee}. We first give some key lemmas in Section \ref{sec:lemmatangent}, and then the proof of Theorem \ref{Rguarantee} is presented in Section \ref{sec:proofmain}.

\subsection{Key Lemmas}\label{sec:lemmatangent}
In this section, we give lemmas that are helpful in the proof of the main theorem. All lemmas are related to the projector $\P_{\S_l}$.  

We first decompose the projector $\P_{\S_l}$ into the sum of products of projectors. As presented in Section \ref{sec:algrgrad}, $\P_{\S_l}\mathcal{G}_l$ is calculated via formulas \eqref{eq:projtan}\eqref{eq:dBinprojtan}\eqref{eq:dVinprojtan}. Let us now re-express components in \eqref{eq:projtan} in terms of products of projections. 
\begin{itemize}
\item By \eqref{eq:projtan}\eqref{eq:dBinprojtan}, the first component in $\P_{\S_l}\mathcal{G}_l$ is
$$
\mathcal{D}_l\times_{i\in[d]}\U_l^{(i)}
=\left(\mathcal{G}_l\times_{i\in[d]}(\U_l^{(i)})^T\right)\times_{i\in[d]}\U_l^{(i)}
=\mathcal{G}_l\times_{i\in[d]}(\U_l^{(i)}(\U_l^{(i)})^T).
$$
We define a projection $\P_{\U_l^{(i)}}^{(i)}~:~\R^{n_1\times\ldots\times n_d}\to\R^{n_1\times\ldots\times n_d}$ by 
$$
\P_{\U_l^{(i)}}^{(i)}\mathcal{Y}=\mathcal{Y}\times_i(\U_l^{(i)}(\U_l^{(i)})^T),\quad
\forall~\mathcal{Y}\in\R^{n_1\times\ldots\times n_d}.
$$
In other words, $\P_{\U_l^{(i)}}^{(i)}$ projects mode-$i$ fibres into the span of $\U_l^{(i)}$. So,
$$
\mathcal{D}_l\times_{i\in[d]}\U_l^{(i)}=\prod_{i=1}^{d}\P_{\U_l^{(i)}}^{(i)}\mathcal{G}_l.
$$
\item By \eqref{eq:projtan}\eqref{eq:dVinprojtan} and \eqref{eq:modeproduct}, the component $\mathcal{C}_l\times_{j\in[d]\setminus i}\U_l^{(j)}\times_{i}\bm{W}_l^{(i)}$ in $\P_{\S_l}\mathcal{G}_l$ satisfies
\begin{equation*}
\left(\mathcal{C}_l\times_{j\in[d]\setminus i}\U_l^{(j)}\times_{i}\bm{W}_l^{(i)}\right)_{(i)}
=\left(\bm{I}-\U_l^{(i)}(\U_l^{(i)})^T\right)(\mathcal{G}_l)_{(i)}\left(\otimes_{j\neq i}\U_l^{(j)}\right)(\mathcal{C}_l)_{(i)}^{\dag}(\mathcal{C}_l)_{(i)}\left(\otimes_{j\neq i}\U_l^{(j)}\right)^T,
\end{equation*}
where we used the following notation
$$
\otimes_{j\neq i}\U_l^{(j)}=\U_l^{(d)}\otimes\ldots\otimes\U_l^{(i+1)}\otimes\U_l^{(i-1)}\otimes\ldots\otimes\U_l^{(1)}.
$$
Define $\Puilp~:~\R^{n_1\times\ldots\times n_d}\to\R^{n_1\times\ldots\times n_d}$ by 
$$
\Puilp\mathcal{Y}=\mathcal{Y}\times_i(\bm{I}-\U_l^{(i)}(\U_l^{(i)})^T),\quad
\forall~\mathcal{Y}\in\R^{n_1\times\ldots\times n_d}.
$$
Thus, $\Puilp$ projects mode-$i$ fibres into the orthogonal complementary of the span of $\U_l^{(i)}$, and 
$$
\big(\Puilp\mathcal{Y}\big)_{(i)}=(\bm{I}-\U_l^{(i)}(\U_l^{(i)})^T)\mathcal{Y}_{(i)}.
$$
Define $\Pujnil~:~\R^{n_1\times\ldots\times n_d}\to\R^{n_1\times\ldots\times n_d}$ by
$$
\big(\Pujnil\mathcal{Y}\big)_{(i)}=\mathcal{Y}_{(i)}\left(\otimes_{j\neq i}\U_l^{(j)}\right)(\mathcal{C}_l)_{(i)}^{\dag}(\mathcal{C}_l)_{(i)}\left(\otimes_{j\neq i}\U_l^{(j)}\right)^T,\quad
\forall~\mathcal{Y}\in\R^{n_1\times\ldots\times n_d}.
$$
Again $\Pujnil$ is a projection. Obviously, 
$$
\mathcal{C}_l\times_{j\in[d]\setminus i}\U_l^{(j)}\times_{i}\bm{W}_l^{(i)}=\Pujnil\Puilp\mathcal{G}_l.
$$
\end{itemize}
Altogether, we can decompose $\P_{\S_l}$ into the sum of products of projectors as follows
\begin{equation}\label{eq:tanprojdec}
\P_{\S_l}=\prod_{i=1}^{d}\Puil+\sum_{i=1}^d\Pujnil\Puilp.
\end{equation}
Furthermore, the projectors in the same product are always commutable. 

Our first lemma estimate the operator norm $\big\|\Puil-\Pui\big\|$.
\begin{lemma}\label{lem:step proj err}  
We have
$$\big\|\Puil-\Pui\big\|\le\frac{1}{\sigma_{r_i}(\T_{(i)})}\|\T-\T_l\|_F,\quad\forall~ i=1,\ldots,d.$$
\end{lemma}
\begin{proof}
Let $i\in\{1,\ldots,d\}$ be given. Then
\begin{equation*}
\begin{split}
\big\|\Puil-\Pui\big\|&=\sup_{\|\Z\|_F=1}\big\|\big(\Puil-\Pui\big)\Z\big\|_F
=\sup_{\|\Z\|_F=1}\big\|\big(\U_l^{(i)}(\U_l^{(i)})^T-\U^{(i)}(\U^{(i)})^T\big)\Z_{(i)}\big\|_F\cr
&\leq\|\U_l^{(i)}(\U_l^{(i)})^T-\U^{(i)}(\U^{(i)})^T\|_2.
\end{split}
\end{equation*}
The upper bound is attainable by choosing $\Z$ such that its mode-$i$ fibres are all multiples of a largest singular vector of $\U_l^{(i)}(\U_l^{(i)})^T-\U^{(i)}(\U^{(i)})^T$ and $\|\Z\|_F=1$. Therefore,
$$
\big\|\Puil-\Pui\big\|=\|\U_l^{(i)}(\U_l^{(i)})^T-\U^{(i)}(\U^{(i)})^T\|_2,
$$
whose right hand side is estimated as follows.

By direct calculation,
$$
\T\times_{j\in[d]\setminus i}(\U^{(j)})^T
=(\mathcal{C}\times_{j\in[d]}\U^{(j)})\times_{j\in[d]\setminus i}(\U^{(j)})^T
=(\mathcal{C}\times_{j\in[d]\setminus i}(\U^{(j)})^T\U^{(j)})\times_i(\U^{(i)})^T
=\mathcal{C}\times_i(\U^{(i)})^T.
$$
Therefore, \eqref{eq:modeimul1} implies $\left(\T\times_{j\in[d]\setminus i}(\U^{(j)})^T\right)_{(i)}=\U^{(i)}\mathcal{C}_{(i)}$. Since the multilinear rank of $\T$ is $\bm{r}$, the matrix $\mathcal{C}_{(i)}$ is of full row rank, so that $\mathcal{C}_{(i)}\mathcal{C}_{(i)}^{\dag}=\bm{I}$. Thus
$$
\U^{(i)}=\left(\T\times_{j\in[d]\setminus i}(\U^{(j)})^T\right)_{(i)}\mathcal{C}_{(i)}^{\dag}.
$$
This together with an equality from the standard textbook \cite{golub2012matrix} implies
\begin{equation}\label{eq:lem:projerr1}
\begin{split}
\|\U_l^{(i)}(\U_l^{i})^T-\U^{(i)}(\U^{(i)})^T\|_2
&=\|(\bm{I}-\U_l^{(i)}(\U_l^{(i)})^T)\U^{(i)}(\U^{(i)})^T\|_2\cr
&=\big\|(\bm{I}-\U_l^{(i)}(\U_l^{(i)})^T)\left(\T\times_{j\in[d]\setminus i}(\U^{(j)})^T\right)_{(i)}\mathcal{C}_{(i)}^{\dagger}(\U^{(i)})^T\big\|_2\cr
&=\big\|\left(\T\times_i(\bm{I}-\U_l^{(i)}(\U_l^{(i)})^T)\times_{j\in[d]\setminus i}(\U^{(j)})^T\right)_{(i)}\mathcal{C}_{(i)}^{\dagger}(\U^{(i)})^T\big\|_2
\end{split}
\end{equation}
Obviously, $\T_l\times_i(\bm{I}-\U_l^{(i)}(\U_l^{(i)})^T)=\bm{0}$. Plugging it into the last line of \eqref{eq:lem:projerr1} gives 
\begin{equation}\label{eq:lem:projerr2}
\begin{split}
\|\U_l^{(i)}&(\U_l^{(i)})^T-\U^{(i)}(\U^{(i)})^T\|_2
=\|((\T-\T_l)\times_i(\bm{I}-\U_l^{(i)}(\U_l^{(i)})^T)\times_{j\in[d]\setminus i}(\U^{(j)})^T)_{(i)}\mathcal{C}_{(i)}^{\dagger}(\U^{(i)})^T\|_2\cr
&\leq\left\|\left((\T-\T_l)\times_i(\bm{I}-\U_l^{(i)}(\U_l^{(i)})^T)\times_{j\in[d]\setminus i}(\U^{(j)})^T\right)_{(i)}\right\|_2\|\mathcal{C}_{(i)}^{\dagger}\|_2\|(\U^{(i)})^T\|_2.
\end{split}
\end{equation}
In view of \eqref{eq:modeproduct},
\begin{equation}\label{eq:lem:projerr3}
\begin{split}
&\left\|\left((\T-\T_l)\times_i(\bm{I}-\U_l^{(i)}(\U_l^{(i)})^T)\times_{j\in[d]\setminus i}(\U^{(j)})^T\right)_{(i)}\right\|_2\cr
=&\|(\bm{I}-\U_l^{(i)}(\U_l^{(i)})^T)\left(\T-\T_l\right)_{(i)}((\U^{(d)})^T\otimes\cdots\otimes (\U^{(i+1)})^T\otimes (\U^{(i-1)})^T\otimes\cdots\otimes (\U^{(1)})^T)^T\|_2\cr
\leq&\|(\bm{I}-\U_l^{(i)}(\U_l^{(i)})^T)\|_2\|\left(\T-\T_l\right)_{(i)}\|_2\cdot\prod_{\substack{j=1\\j\neq i}}^{d}\|(\U^{(j)})^T\|_2=\|\left(\T-\T_l\right)_{(i)}\|_2.
\end{split}
\end{equation}
By combining \eqref{eq:lem:projerr2} and \eqref{eq:lem:projerr3}, we obtain
$$
\|\U_l^{(i)}(\U_l^{(i)})^T-\U^{(i)}(\U^{(i)})^T\|_2
\leq\|(\T-\T_l)_{(i)}\|_2\|\mathcal{C}_{(i)}^{\dagger}\|_2
\leq \|\mathcal{C}_{(i)}^{\dagger}\|_2\|(\T-\T_l)_{(i)}\|_F
=\frac{1}{\sigma_{r_i}(\T_{(i)})}\|\T-\T_l\|_F.
$$ 
\end{proof}

Next, we are going to estimate $\norm{(\I-\P_{\S_l})\T}_F$.
\begin{lemma}\label{lem:proj err} 
We have
\begin{align*}
\norm{(\I-\P_{\S_l})\T}_F\leq \frac{2^d-1}{\min\limits_{i}(\sigma_{r_i}(\T_{(i)}))}\norm{\T-\T_l}_F^2
\end{align*}
\end{lemma}
\begin{proof}
By the decomposition of $\P_{\S_l}$, we have
\begin{equation*}
\begin{split}
(\I-&\P_{\S_l})\T
=\Bigg(\I-\Big(\prod_{i=1}^d\Puil + \sum_{i=1}^d\Pujnil\Puilp\Big)\Bigg)\T\cr
&=\Bigg(\prod_{i=1}^d\bigg(\Puil+\Puilp\bigg)-\bigg(\prod_{i=1}^d\Puil + \sum_{i=1}^d\Pujnil\Puilp\bigg)\Bigg)\T\cr
&=\Bigg(\sum_{i=1}^d\Big(\Puilp\prod_{j\neq i}\P_{\U_l^{(j)}}^{(i)}\Big)+\sum_{i=1}^d\sum_{j\neq u}\Big(\P_{\U_l^{{(i)}\perp}}^{(i)}\P_{\U_l^{{(j)}\perp}}^{(j)}\prod_{k\neq i,j}\P_{\U_l^{(k)}}^{(k)}\Big)+\cdots+\prod_{i=1}^d\Puilp\cr
&\qquad\qquad\qquad\qquad-\sum_{i=1}^d\Pujnil\Puilp\Bigg)\T\cr
&=\sum_{i=1}^d\Puilp\Bigg(\Big(\prod_{j\neq i}\P_{\U_l^{(j)}}^{(j)}-\Pujnil\Big)+\sum_{j\neq i}\Big(\P_{\U_l^{{(j)}\perp}}^{(j)}\prod_{k\neq i,j}\P_{\U_l^{(k)}}^{(k)}\Big)+\cdots+\prod_{j\neq i}\P_{\U_l^{(j)\perp}}^{(j)}\Bigg)\T\cr
&=\sum_{i=1}^d\Bigg(\Big(\prod_{j\neq i}\P_{\U_l^{(j)}}^{(j)}-\Pujnil\Big)+\sum_{j\neq i}\Big(\P_{\U_l^{{(j)\perp}}}^{(j)}\prod_{k\neq i,j}\P_{\U_l^{(k)}}^{(k)}\Big)+\cdots+\prod_{j\neq i}\P_{\U_l^{(j)\perp}}^{(j)}\Bigg)\Puilp\T
\end{split}
\end{equation*}
Note that, for any $i$,
$$
\Puilp\T=\T-\Puil\T=\Pui\T-\Puil\T=\Big(\Pui-\Puil\Big)\T,
$$
from which it follows
\begin{equation}\label{eq:lem:projerr10}
\begin{split}
&(\I-\P_{\S_l})\T\cr
=&\sum_{i=1}^d\Bigg(\Big(\prod_{j\neq i}\P_{\U_l^{(j)}}^{(j)}-\Pujnil\Big)+\sum_{j\neq i}\Big(\P_{\U_l^{{(j)}\perp}}^{(j)}\prod_{k\neq i,j}\P_{\U_l^{(k)}}^{(k)}\Big)+\cdots+\prod_{j\neq i}\P_{\U_l^{(j)}\perp}^{(j)}\Bigg)
\Big(\Pui-\Puil\Big)\T\cr
=&\sum_{i=1}^d\Big(\Pui-\Puil\Big)\Bigg(\Big(\prod_{j\neq i}\P_{\U_l^{(j)}}^{(j)}-\Pujnil\Big)+\sum_{j\neq i}\Big(\prod_{k\neq i,j}\P_{\U_l^{(k)}}^{(k)}\Big)\P_{\U_l^{{(j)}\perp}}^{(j)}+\cdots+\prod_{j\neq i}\P_{\U_l^{(j)\perp}}^{(j)}\Bigg)\T.
\end{split}
\end{equation}
Because $\P_{\U_l^{(j)\perp}}^{(j)}\T_l=0$ for any $j$, we have, for any $i$,
\begin{equation}\label{eq:lem:projerr20}
\Bigg(\sum_{j\neq i}\Big(\prod_{k\neq i,j}\P_{\U_l^{(k)}}^{(k)}\Big)\P_{\U_l^{{(j)}\perp}}^{(j)}+\cdots+\prod_{j\neq i}\P_{\U_l^{(j)\perp}}^{(j)}\Bigg)\T_l=0
\end{equation}
Moreover, for any $i$, 
\begin{equation}\label{eq:lem:projerr30}
\Big(\prod_{j\neq i}\P_{\U_l^{(j)}}^{(j)}-\Pujnil\Big)\T_l
=\prod_{j\neq i}\P_{\U_l^{(j)}}^{(j)}\T_l-\Pujnil\T_l
=\T_l-\T_l=0.
\end{equation}
Combined with \eqref{eq:lem:projerr20} and \eqref{eq:lem:projerr30}, Eq. \eqref{eq:lem:projerr10} gives
\begin{equation*}
\begin{split}
(\I-\P_{\S_l})\T=\sum_{i=1}^d\Big(\Pui-&\Puil\Big)\Bigg(\Big(\prod_{j\neq i}\P_{\U_l^{(j)}}^{(j)}-\P_{\{\U_l^{(j)}\}_{j\ne i}}^{(j\ne i)}\Big)\cr
&\qquad+\sum_{j\neq i}\Big(\prod_{k\neq i,j}\P_{\U_l^{(k)}}^{(k)}\Big)\P_{\U_l^{{(j)}\perp}}^{(j)}+\cdots+\prod_{j\neq i}\P_{\U_l^{(j)\perp}}^{(j)}\Bigg)(\T-\T_l).
\end{split}
\end{equation*}
Therefore, 
\begin{equation*}
\begin{split}
\|(\I-\P_{\S_l})\T\|_F
\leq &\sum_{i=1}^d\Big\|\Pui-\Puil\Big\|\Bigg(\Big\|\prod_{j\neq i}\P_{\U_l^{(j)}}^{(j)}-\Pujnil\Big\|\cr
&\qquad\qquad\qquad+\sum_{j\neq i}\prod_{k\neq i,j}\big\|\P_{\U_l^{(k)}}^{(k)}\big\|\big\|\P_{\U_l^{{(j)}\perp}}^{(j)}\big\|+\cdots+\prod_{j\neq i}\big\|\P_{\U_l^{(j)\perp}}^{(j)}\big\|\Bigg)\|\T-\T_l\|_F\cr
\leq&\sum_{i=1}^d\frac{1}{\sigma_{r_i}(\T_{(i)})}\|\T-\T_l\|_F
\Bigg(1+\sum_{j\neq i}1+\ldots+1\Bigg)\|\T-\T_l\|_F\cr
\leq&\frac{1}{\min_i(\sigma_{r_i}(\T_{(i)}))}\|\T-\T_l\|_F^2 \cdot \sum_{i=1}^d\Bigg(1+\sum_{j\neq i}1+\ldots+1\Bigg)\cr
=&\frac{2^d-1}{\min_i(\sigma_{r_i}(\T_{(i)}))}\|\T-\T_l\|_F^2,
\end{split}
\end{equation*}
where in the second inequality we have used Lemma \ref{lem:step proj err} and the fact $\Big\|\prod_{j\neq i}\P_{\U_l^{(j)}}^{(j)}-\Pujnil\Big\|=1$ because the operator is a projector by direct calculation.
\end{proof}

We remark that we can take advantage of higher order of complementary projections to make the bound tighter. In particular, the bound in Lemma \ref{lem:proj err} can be improved to 
$$\bigg(\Big(1+\frac{1}{\min\limits_i(\sigma_{r_i}(\T_{(i)}))}\|\T-\T_l\|_F\Big)^d-\frac{d}{\min\limits_i(\sigma_{r_i}(\T_{(i)}))}\|\T-\T_l\|_F-1\bigg)\|\T-\T_l\|_F.
$$ 
However, it is not necessary in this paper.

The last lemma uses TRIP constant to estimate $\norm{\P_{\S_l}\A^*\A(\I-\P_{\S_l})\T}_F$.
\begin{lemma}\label{lem:sensing dense2}
Assume $\A$ satisfies TRIP with constant $\delta_{2\bm{r}}$.
Then, it holds that
$$
\norm{\P_{\S_l}\A^*\A(\I-\P_{\S_l})\T}_F\leq(1+\delta_{2\rr})\norm{(\I-\P_{\S_l})\T}_F.
$$
\end{lemma}
\begin{proof}
We first show that the multilinear rank of $(\I-\P_{\S_l})\T$ is at most $2\bm{r}$. It suffices to show that $\mathrm{span}\big(((\I-\P_{\S_l})\T)_{(k)}\big)\subset\mathrm{span}([\U^{(k)}, \U_l^{(k)}])$ for all $k$.  To this end, similar to formulas \eqref{eq:projtan}\eqref{eq:dBinprojtan}\eqref{eq:dVinprojtan} we obtain
$$
\P_{\S_l}\T=\underbrace{(\T\times_{i\in[d]}(\U_{l}^{(i)})^T)\times_{i\in[d]}\U_{l}^{(i)}}_{\mathcal{B}_0}+\sum_{i=1}^{d}\underbrace{\C_l\times_{j\in[d]\setminus i}\U_l^{(j)}\times_{i}\bm{W}^{(i)}}_{\mathcal{B}_i},
$$
where
$$
\bm{W}^{(i)}=\left(\bm{I}-\U_l^{(i)}(\U_l^{(i)})^T\right)\T_{(i)}\left(\otimes_{j\neq i}\U_l^{(j)}\right)(\mathcal{C}_l)_{(i)}^{\dag}.
$$
Since $\T=\C\times_{i\in[d]}\U^{(i)}$, we have
$\mathcal{B}_0=\C\times_{i\in[d]}(\U_l^{(i)}(\U_l^{(i)})^T\U^{(i)})$, which implies 
$$
\mathrm{span}\big((\mathcal{B}_0)_{(k)}\big)\subset\mathrm{span}(\U_l^{(k)}(\U_l^{(k)})^T\U^{(k)}).
$$
Moreover, a simple calculation leads to $\bm{W}^{(i)}=(\bm{I}-\U_l^{(i)}(\U_l^{(i)})^T\bm{U}^{(i)}\mathcal{C}_{(i)}\left(\otimes_{j\neq i}\U^{(j)}\right)^T\left(\otimes_{j\neq i}\U_l^{(j)}\right)(\mathcal{C}_l)_{(i)}^{\dag}$, and thus
$$
\mathrm{span}\big((\mathcal{B}_i)_{(k)}\big)\subset
\begin{cases}
\mathrm{span}((\bm{I}-\U_l^{(k)}(\U_l^{(k)})^T)\U^{(k)}),&\mbox{if }k=i\cr
\mathrm{span}(\U_l^{(k)}),&\mbox{if }k\neq i.
\end{cases}
$$
Obviously, $\mathrm{span}(\T_{(k)})=\mathrm{span}(\U^{(k)})$.
Therefore, for any $k$,
\begin{equation*}
\begin{split}
\mathrm{span}\big(((\I-\P_{\S_l})\T)_{(k)}\big)
&\subset\mathrm{span}(\U^{(k)})\oplus\mathrm{span}(\U_l^{(k)}(\U_l^{(k)})^T\U^{(k)})\oplus\mathrm{span}((\bm{I}-\U_l^{(k)}(\U_l^{(k)})^T)\U^{(k)})\oplus\mathrm{span}(\U_l^{(k)})\cr
&\subset\mathrm{span}([\U^{(k)}, \U_l^{(k)}])
\end{split}
\end{equation*}

Now, we have 
\begin{align*}
\norm{\P_{\S_l}\A^*\A(\I-\P_{\S_l})\T}_F & =\sup_{\norm{\Z}_F=1}\abs{\inner{\P_{\S_l}\A^*\A(\I-\P_{\S_l})\T,\Z}}=\sup_{\norm{\Z}_F=1}\abs{\inner{\A(\I-\P_{\S_l})\T, \A\P_{\S_l}\Z}}\\
&\leq \|\A(\I-\P_{\S_l})\T\|_2 \|\A\P_{\S_l}\Z\|_2 \leq(1+\delta_{2\rr})\|(\I-\P_{\S_l})\T\|_F
\end{align*}
where the last inequality follows from TRIP assumption together with the facts that $\mathrm{mulrank}((\I-\P_{\S_l})(\T))\preceq 2\rr$ and $\mathrm{mulrank}(\P_{\S_l}(\Z))\preceq 2\rr$.
\end{proof}

\subsection{Proof of Theorem \ref{Rguarantee}}\label{sec:proofmain}
Now we are ready to prove the main result Theorem \ref{Rguarantee}.

\begin{proof}[Proof of Theorem 	\ref{Rguarantee}]
By the notations in Algorithm \ref{alg:Rgrad}, $\T_{l+1}=\mathscr{H}_{\bm{r}}(\mathcal{W}_l)$, from which \eqref{eq:quasibest} derives 
$$
\|\T_{l+1}-\mathcal{W}_l\|_F
\leq\sqrt{d}\|\P_{\mathbb{M}_{\bm{r}}}(\mathcal{W}_l)-\mathcal{W}_l\|_F
\leq\sqrt{d}\|\T-\mathcal{W}_l\|_F.
$$
Therefore,
$$\|\T_{l+1}-\T\|_F\le\|\T_{l+1}-\mathcal{W}_l\|_F+\|\mathcal{W}_l-\T\|_F\le(\sqrt{d}+1)\|\mathcal{W}_l-\T\|_F.
$$
Substituting $\mathcal{W}_l=\T_l-\alpha_l\P_{\mathcal{S}_l}\mathcal{G}_l$ into the above inequality gives
\begin{equation*}
\begin{split}
\norm{\T_{l+1}-\T}_F\le&(\sqrt{d}+1)\|(\T_l-\alpha_l\P_{\mathcal{S}_l}\mathcal{G}_l)-\T\|_F=(\sqrt{d}+1)\|\T_l-\alpha_l\P_{\S_l}\A^{*}(\A\T_l-\bm{y})-\T\|_F\\
=&(\sqrt{d}+1)\|(\mathscr{I}-\alpha_l\P_{\S_l}\A^{*}\A)(\T_l-\T)\|_F\\
\le&(\sqrt{d}+1)\Big(\underbrace{\|(\P_{\S_l}-\alpha_l\P_{\S_l}\A^{*}\A\P_{\S_l})(\T_l-\T)\|_F}_{I_1}
+\underbrace{\|(\mathscr{I}-\P_{\S_l})(\T_l-\T)\|_F}_{I_2}
\cr
&\qquad\qquad\qquad\qquad+
\alpha_l\underbrace{\|\P_{\S_l}\A^{*}\A(\mathscr{I}-\P_{\S_l})(\T_l-\T)\|_F}_{I_3}\Big)
\end{split}
\end{equation*}
In the following, we will bound $I_1$, $I_2$ and $I_3$ one by one.
\begin{itemize}
\item \emph{Bound of $I_1$:} We bound $I_1$ by bounding the operator norm of $\P_{\S_l}-\P_{\S_l}\A^*\A\P_{\S_l}$. Since it is a self-adjoint linear operator, we have 
\begin{align*}
\norm{\P_{\S_l}-\P_{\S_l}\A^*\A\P_{\S_l}} &=\sup_{\norm{\Z}_F=1}\abs{\inner{(\P_{\S_l}-\P_{\S_l}\A^*\A\P_{\S_l})\Z,\Z}}=\sup_{\norm{\Z}_F=1}\abs{\norm{\P_{\S_l}\Z}_F^2-\norm{\A\P_{\S_l}\Z}_2^2}\\
&\le\sup_{\norm{\Z}_F=1}\delta_{2\rr}\norm{\P_{\S_l}\Z}_F^2\le \delta_{2\rr},\numberthis\label{eq:spectral_bound_gniht_dense}
\end{align*}
where the inequality follows from the TRIP assumption of $\A$ by noting that $\mathrm{mulrank}(\P_{\S_l}\Z)\preceq 2\rr$. The TRIP assumption also helps the estimation of the stepsize $\alpha_l$ as
\begin{equation}\label{eq:stepsize_gniht_dense}
\frac{1}{1+\delta_{2\rr}}\leq\alpha_l=\frac{\norm{\P_{\S_l}(\mathcal{G}_l)}_F^2}{\norm{\A\P_{\S_l}(\mathcal{G}_l)}_2^2}\leq\frac{1}{1-\delta_{2\rr}},
\end{equation}
which immediately implies 
\begin{equation}\label{eq:stepsize_gniht_dense1}
\abs{\alpha_l-1}\leq\frac{\delta_{2\rr}}{1-\delta_{2\rr}}.
\end{equation}
Combining \eqref{eq:spectral_bound_gniht_dense} and \eqref{eq:stepsize_gniht_dense1} gives the bound of the operator norm of $\P_{\S_l}-\alpha_l\P_{\S_l}\A^*\A\P_{\S_l}$
\begin{align*}
\norm{\P_{\S_l}-\alpha_l\P_{\S_l}\A^*\A\P_{\S_l}} & \leq \norm{\P_{\S_l}-\P_{\S_l}\A^*\A\P_{\S_l}}+\abs{1-\alpha_l}\norm{\P_{\S_l}\A^*\A\P_{\S_l}}\\
&\leq \delta_{2\rr}+\frac{\delta_{2\rr}}{1-\delta_{2\rr}}(1+\delta_{2\rr}) = \frac{2\delta_{2\rr}}{1-\delta_{2\rr}}\label{eq:operator_bound_gniht_dense}.
\end{align*}
Thus $I_1$ can be bounded as 
\begin{equation}\label{eq:gniht_dense_I1}
I_1\leq\frac{2\delta_{2\rr}}{1-\delta_{2\rr}}\norm{\T_l-\T}_F.
\end{equation}

\item \emph{Bound of $I_2$:} The second term $I_2$ can be bounded by applying Lemma~\ref{lem:proj err} directly as follows
\begin{equation}\label{eq:gniht_dense_I2}
I_2=\|(\mathscr{I}-\P_{\S_l})\T\|_F\leq \frac{2^d-1}{\min\limits_{i}(\sigma_{r_i}(\T_{(i)}))}\norm{\T-\T_l}_F^2
\end{equation}

\item \emph{Bound of $I_3$:} We use Lemmas \ref{lem:proj err} and \ref{lem:sensing dense2} to bound $I_3$ as in below
\begin{align*}
I_3=\|\P_{\S_l}\A^{*}\A(\mathscr{I}-\P_{\S_l})\T\|_F&\leq (1+\delta_{2\rr})\norm{(\mathscr{I}-\P_{\S_l})\T}_F
\leq (1+\delta_{2\rr})\frac{2^d-1}{\min\limits_{i}(\sigma_{r_i}(\T_{(i)}))}\norm{\T-\T_l}_F^2
\end{align*}

\end{itemize} 
Combining bounds of $I_1$, $I_2$, $I_3$ and $\alpha_l$ gives
\begin{align}
\norm{\T_{l+1}-\T}_F\leq
\frac{2(\sqrt{d}+1)}{1-\delta_{2\bm{r}}}
\left(\delta_{2\rr}+\frac{2^d-1}{\min\limits_{i}(\sigma_{r_i}(\T_{(i)}))}\norm{\T-\T_l}_F\right)\norm{\T_l-\T}_F.
\label{eq:recursive_gniht_dense_null}
\end{align}

Now we estimate the initial error $\|\T-\T_0\|_F$. As in Algorithm \ref{alg:Rgrad}, $\T_0=\mathscr{H}_{\rr}(\A^*\bm{y})$ is the result of one step of IHT with step-size $1$ starting from the zero tensor.  Let $\bm{Q}^{(i)}\in\R^{n_i\times 2r_i}$ be the orthogonal matrix which spans the column subspaces of mode-$i$ matricizations $(\T_0)_{(i)}$ and $\T_{(i)}$ for $i=1,\ldots,d$. The first $r_i$ columns of $\bm{Q}^{(i)}$ are the leading $r_i$ singular vectors of $(\A^*\bm{y})_{(i)}$. Define projection operator
\begin{equation*}
\P_{\bm{Q}} = \prod_{i=1}^d\P_{\bm{Q^{(i)}}}^{(i)},
\end{equation*}
whose range is a subset of tensors with a multilinear rank at most $2\rr$. Since $\bm{Q}^{(i)}$ contains the leading $r_i$ singular vectors of $(\A^*\bm{y})_{(i)}$ for all $i$ and according to \cite[Theorem 2]{de2000multilinear}, core tensor of  $\A^*\bm{y}$ has the ordering property, $\P_{\bm{Q}}\A^*\bm{y}$ and $\A^*\bm{y}$ have the same leading singular vectors and singular values. Therefore, we have $\T_0=\mathscr{H}_{\bm{r}}(\A^*\bm{y})=\mathscr{H}_{\bm{r}}(\P_{\bm{Q}}\A^*\bm{y})$. This together with \eqref{eq:quasibest} implies
$$
\norm{\T_0-\P_{\bm{Q}}\A^*\bm{y}}_F\leq \sqrt{d}\norm{\P_{\mathbb{M}_{\bm{r}}}(\P_{\bm{Q}}\A^*\bm{y})-\P_{\bm{Q}}\A^*\bm{y}}_F
\leq\sqrt{d}\norm{\T-\P_{\bm{Q}}\A^*\bm{y}}_F.
$$
Additionally, we have $\P_{\bm{Q}}\T=\T$. So we have 
\begin{align}
\norm{\T_0-\T}_F&\leq \norm{\T_0-\P_{\bm{Q}}\A^*\bm{y}}_F+\norm{\P_{\bm{Q}}\A^*\bm{y}-\T}_F\leq (\sqrt{d}+1)\norm{\T-\P_{\bm{Q}}\A^*\bm{y}}_F \notag\\
&= (\sqrt{d}+1)\norm{(\P_{\bm{Q}}-\P_{\bm{Q}}\A^*\A\P_{\bm{Q}})\T}_F \leq(\sqrt{d}+1)\norm{(\P_{\bm{Q}}-\P_{\bm{Q}}\A^*\A\P_{\bm{Q}})}\norm{\T}_F\notag\\
&\leq (\sqrt{d}+1)\delta_{2\rr}\norm{\T}_F,
\label{eq:gniht_dense_init_bound}
\end{align}
where the last inequality follows from $\norm{\P_{\bm{Q}}-\P_{\bm{Q}}\A^*\A\P_{\bm{Q}}}\leq\delta_{2\bm{r}}$ obtained similarly to \eqref{eq:spectral_bound_gniht_dense}.

Define $\gamma$ as in \eqref{eq:gamma}, i.e.,
\begin{equation}\label{eq:gamma2}
\gamma=\frac{2\delta_{2\bm{r}}}{1-\delta_{2\bm{r}}}
(\sqrt{d}+1)\left(1+(2^d-1)(\sqrt{d}+1)\frac{\norm{\T}_F}{\min\limits_{i}(\sigma_{r_i}(\T_{(i)}))}\right).
\end{equation}
If $\gamma<1$, inserting~\eqref{eq:gamma2} into \eqref{eq:recursive_gniht_dense_null} and proof by induction gives 
\begin{equation}\label{eq:recursive_gniht_dense}
\norm{\T_{l+1}-\T}_F \leq \gamma\norm{\T_l-\T}_F\leq\ldots
\leq\gamma^{l+1}\norm{\T_0-\T}_F.
\end{equation}
Moreover, if \eqref{eq:ric} is satisfied, then, letting $i_0=\arg\min_{i}\sigma_1(\T_{(i)})$,
\begin{equation*}
\begin{split}
\gamma&\leq\frac{2\delta}{1-\delta}(\sqrt{d}+1)\left(1+(2^d-1)(\sqrt{d}+1)\frac{\norm{\T}_F}{\min\limits_{i}(\sigma_{r_i}(\T_{(i)}))}\right)
\leq \frac{2\delta}{1-\delta}2^d(\sqrt{d}+1)^2\frac{\norm{\T}_F}{\min\limits_{i}(\sigma_{r_i}(\T_{(i)}))}\cr
&\leq\frac{2\delta}{1-\delta}2^d(\sqrt{d}+1)^2\frac{\sqrt{r_{i_0}}\|\T_{(i_0)}\|_2}{\min\limits_{i}(\sigma_{r_i}(\T_{(i)}))}
\leq\frac{2\delta}{1-\delta}2^d(\sqrt{d}+1)^2\sqrt{r}\kappa
=\frac{2}{3}\frac{1}{1-\delta}
\leq \frac{2}{3}\cdot\frac{48}{47}<1,
\end{split}
\end{equation*} 
where we have used $\|\T\|_F=\norm{\T_{(i)}}_F\leq\sqrt{r_i}\norm{\T_{(i)}}_2$ for all $i$ in the first inequality in the second line.
\end{proof}

\section{Numerical Experiments}
In this section, we present numerical results of Algorithm \ref{alg:Rgrad} for solving the tensor recovery problem \eqref{eq:tensorrecovery}. We focus on the demonstration of the sampling complexity of Algorithm \ref{alg:Rgrad} as shown in the main results of this paper. The computational efficiency of this algorithm has been already illustrated in \cite{kressner2014low}. 

We test only cubic tensors of size $n\times n\times n$ with multilinear rank $\bm{r}=(r_1,r_2,r_3)$, which are generated randomly through Tucker decomposition. The operator $\A$ is drawn randomly from different distributions, described in detail later. For each set of parameters $n,\bm{r},m$, we run $20$ random tests and count the success rate. A test is regarded as a successful recovery if the relative error of the recovered tensor $\T_l$ and the original tensor $\T$ satisfies $\|\T_l-\T\|_F/\|\T\|_F\leq 10^{-3}$. 

We first show results where the measurement tensors $\mathcal{A}_i$ in $\A$ are randomly drawn from Gaussian distribution. In particular, the entries of $\mathcal{A}_i$ are drawn from i.i.d. Gaussian distribution with mean $0$ and variance $1/m$. Figure \ref{fig:trans_1} shows the results. In Figure~\ref{fig:trans_1}(a)(c),  we plot the curve of successful recovery rate against the sampling ratio $m/n^3$ for tensors of size $10\times 10\times 10$ and different ranks. Figure~\ref{fig:trans_1}(e) depicts the successful recovery rate under different tensor sizes $n$ and different number of samples $m$ for tensors with a fixed multilinear rank $\bm{r}=(7,7,7)$. The color of each cell reflects the empirical recovery rate ranging from $0$ to $1$. Black cell means exact recovery in all experiments and white cell means all experiments failed. We see from this figure that the minimum $m$ for a nearly $100\%$ successful recovery grows linearly with $n$, which is consistent with our results in Corollary \ref{cor:SamComplexity}.

\begin{figure}[ht!]
\centering
\subfigure[]{\includegraphics[width=.45\linewidth,height=4.2cm]{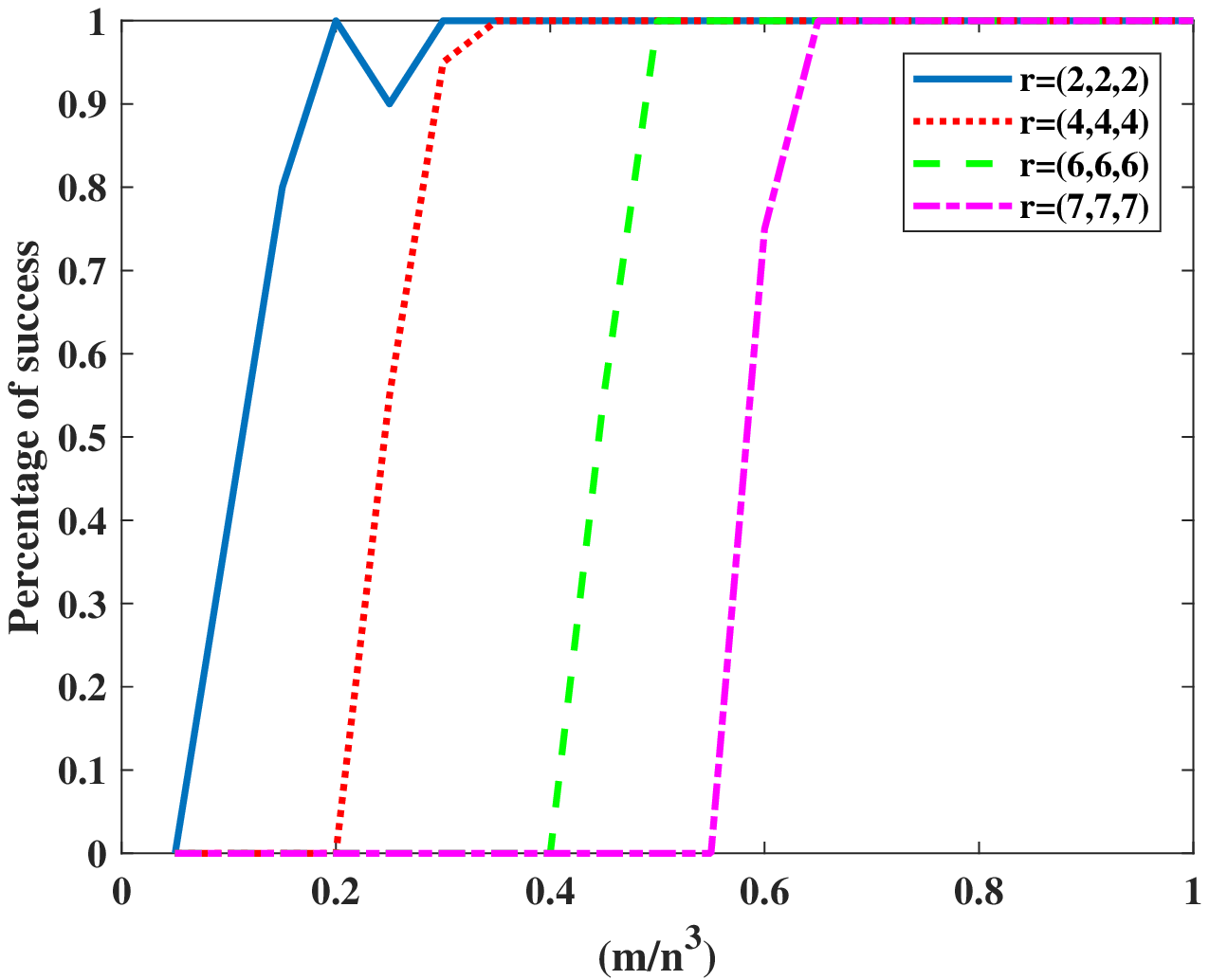}}
\subfigure[]{\includegraphics[width=.45\linewidth,height=4.2cm]{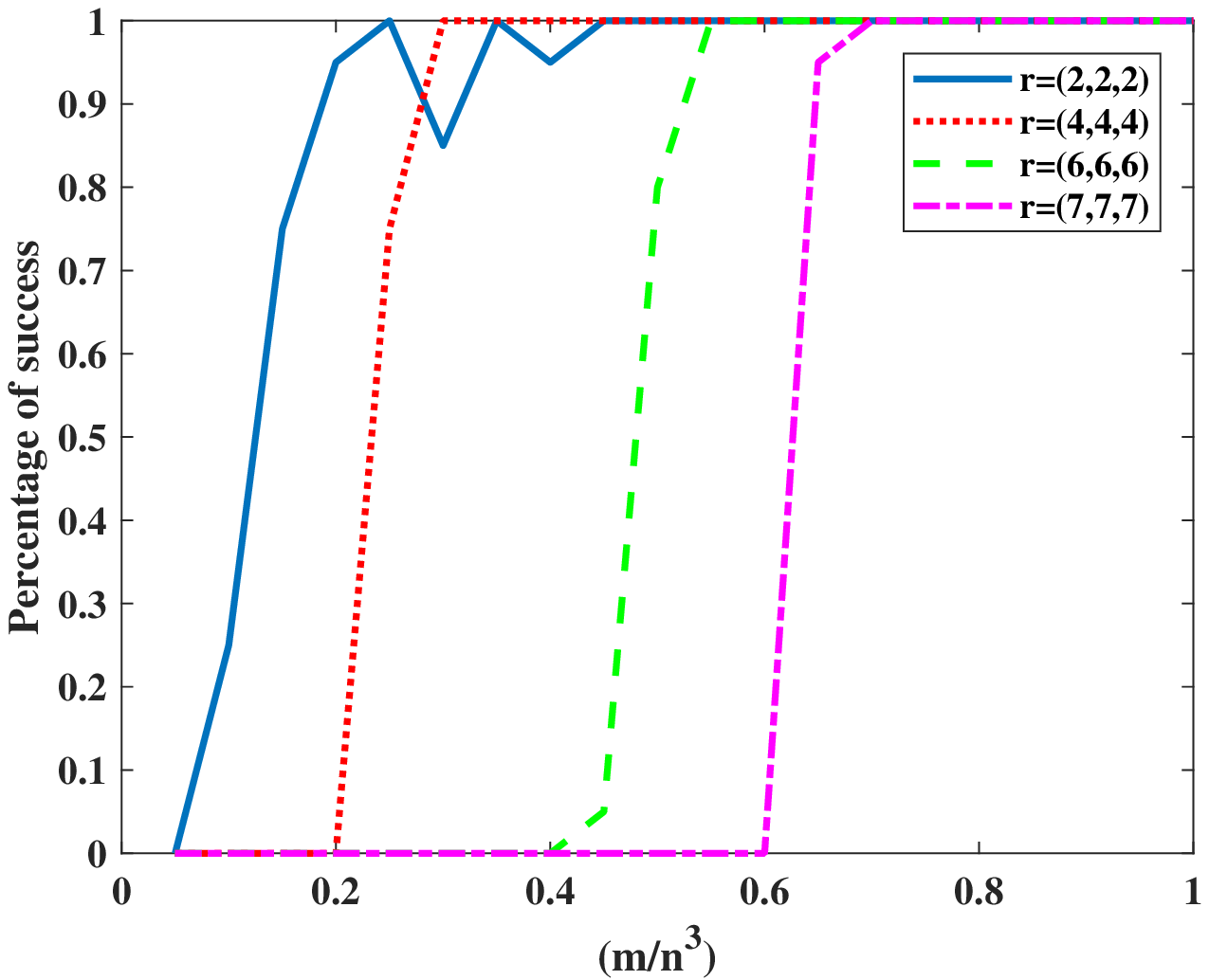}} \\
\subfigure[]{\includegraphics[width=.45\linewidth,height=4.2cm]{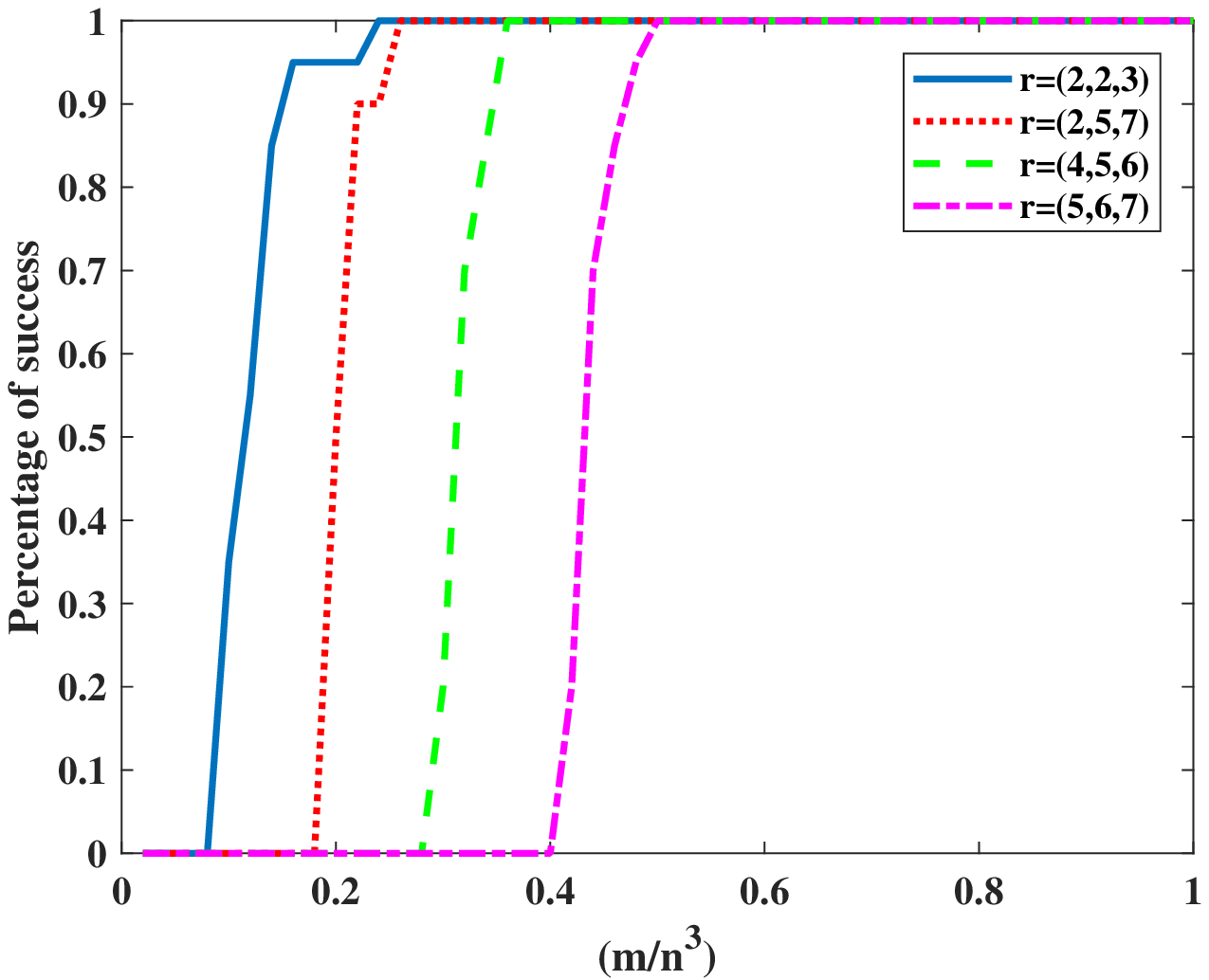}}
\subfigure[]{\includegraphics[width=.45\linewidth,height=4.2cm]{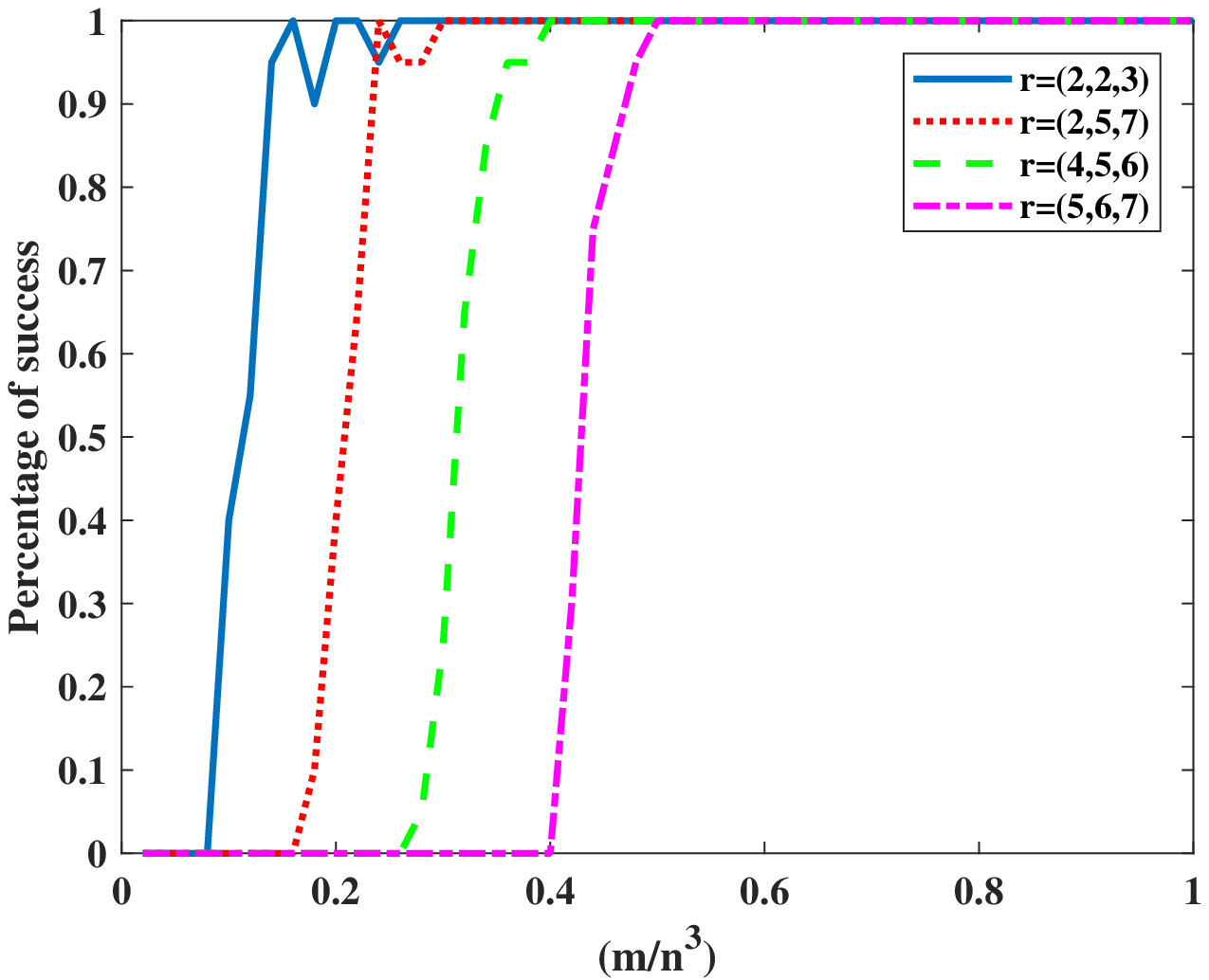}} \\
\subfigure[]{\includegraphics[width=.45\linewidth,height=4.2cm]{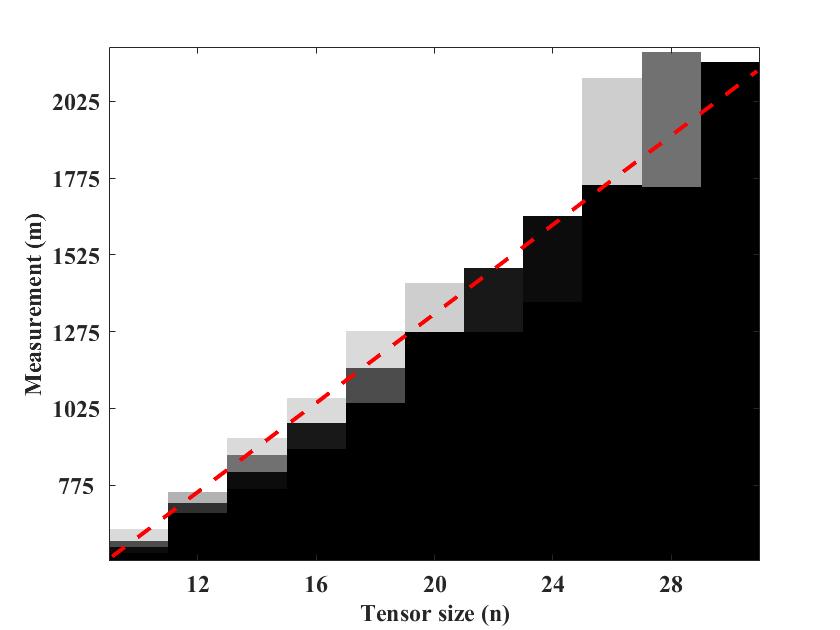}}
\subfigure[]{\includegraphics[width=.45\linewidth,height=4.2cm]{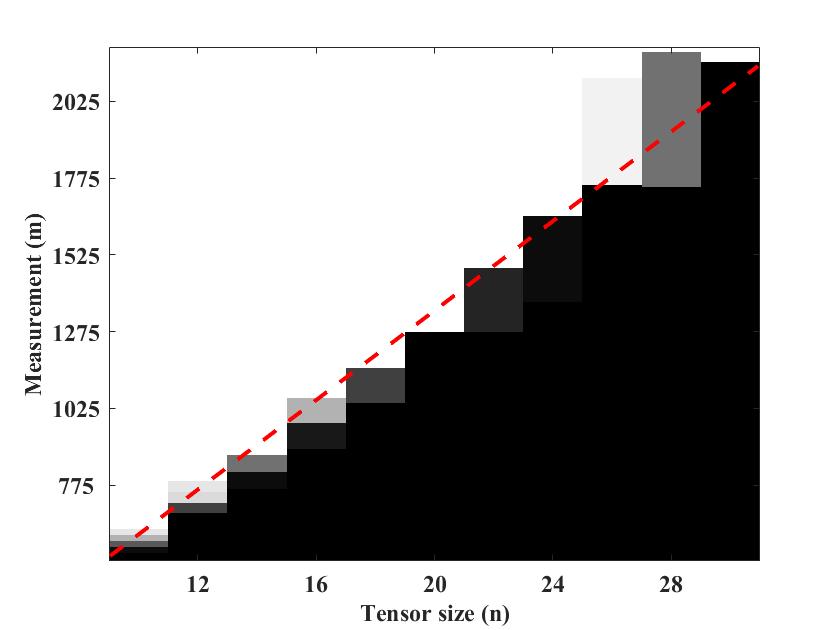}}
\caption{Results of tensor recovery from Gaussian measurements. (a) Recovery of $10\times 10\times 10$ tensors of different ranks from exact data. (b) Recovery of $10\times 10\times 10$ tensors of different ranks from noisy data. (c) Recovery of $10\times 10\times 10$ tensors of different ranks from exact data. (d) Recovery of $10\times 10\times 10$ tensors of different ranks from noisy data. (e) Successful recovery rate of tensors with a fixed rank $\bm{r}=(7,7,7)$ from exact data. (f) Successful recovery rate of tensors with a fixed rank $\bm{r}=(7,7,7)$ from noisy data. The red dashed line is the graph of a linear function of $n$.}
\label{fig:trans_1}
\end{figure}

To demonstrate the robustness of the algorithm, we also investigate the recovery efficiency in the presence of noise. The linear measurement $\bm{y}$ is perturbed by a Gaussian noise whose $2$-norm is $10^{-4}$ of $\|\bm{y}\|_2$. The recovery is regarded as successful when the recovered tensor has a relative error $10^{-3}$. The results are shown in Figure~\ref{fig:trans_1}(b), (d) and (f). Again, we see that, even in the presence of noise, the emprical sampling complexity $m$ of our algorithm is linear in $n$.

Finally, we illustrate results when $\A$ is generated from a random Fourier model in \cite{rauhut2017low,rauhut2015tensor}. In particular, $\A=\frac{1}{\sqrt{m}}\mathscr{R}_{\Omega}\mathscr{F}_d\mathscr{D}$ is the composition of a random sign flip map $\mathscr{D}$ with independent $\pm{1}$ Rademacher variables, a $d$-dimensional Fourier transform $\mathscr{F}_d$, and a random subsampling operator $\mathscr{R}_{\Omega}$ that takes only entries on $\Omega$. The random Fourier operator can sample large size tensors, as the sampling operator here contains much simpler  parameters than in Gaussian measurements. The results with random Fourier measurements are shown in Figure~\ref{fig:trans_2}.  Figure~\ref{fig:trans_2}(a) and (b) presents the curve of successful recovery rate against the sampling ratio $m/n^3$ for tensors of size $50\times 50\times 50$ with different ranks. 
Figure~\ref{fig:trans_2}(c) shows the successful recovery rate under different tensor sizes $n$ and different number of samples $m$ for tensors with a fixed multilinear rank $\bm{r}=(35,35,35)$. Again, the gray level of each cell reflects the empirical recovery rate ranging from $0$ to $1$. We see from this figure that the minimum $m$ for a nearly $100\%$ successful recovery grows linearly with $n\log^2(n)$, which is consistent with the combination of our main result Theorem \ref{Rguarantee} and the result in \cite{rauhut2017low} on TRIP of random Fourier operators.

\begin{figure}[ht!]
\centering
\subfigure[]{\includegraphics[width=.45\linewidth,height=4.2cm]{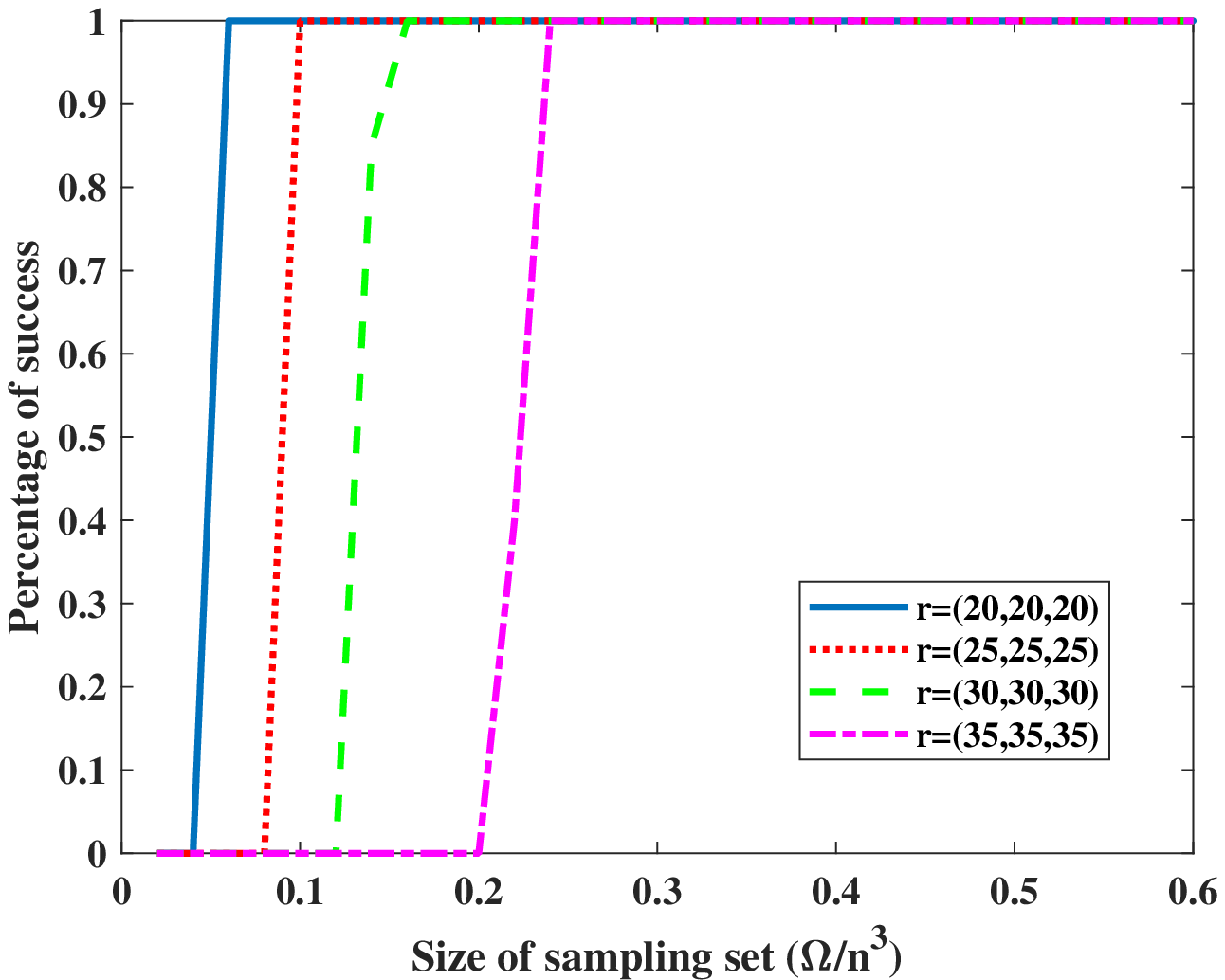}}
\subfigure[]{\includegraphics[width=.45\linewidth,height=4.2cm]{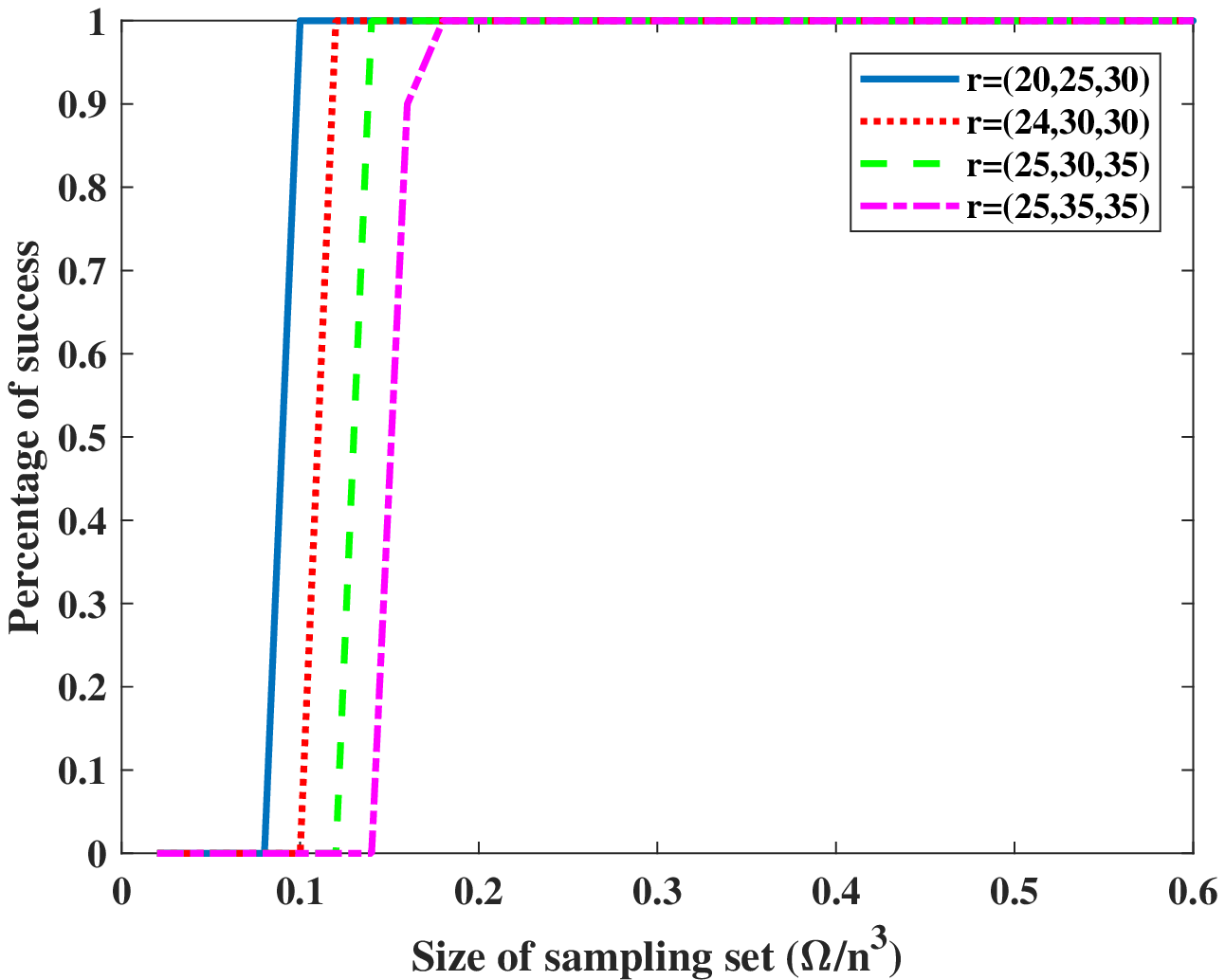}} \\
\subfigure[]{\includegraphics[width=.45\linewidth,height=4.3cm]{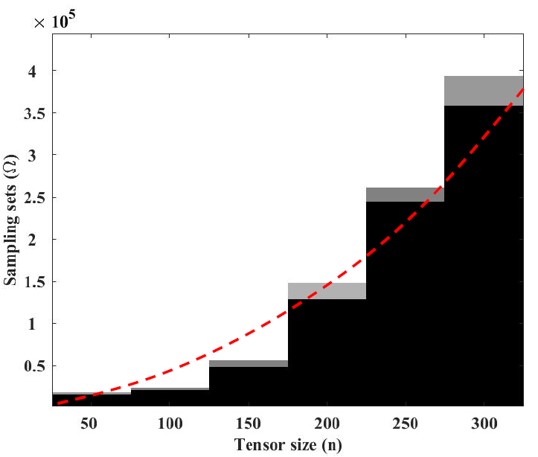}}
\caption{Results of tensor recovery from random Fourier measurements. (a) Recovery of $50\times 50\times 50$ tensors of different ranks. (b) Recovery of $50\times 50\times 50$ tensors of different ranks. (c) Successful recovery rate of tensors with a fixed rank $\bm{r}=(35,35,35)$. The red dashed line shows the graph of a function proportional to $n\log^2(n)$.}
\label{fig:trans_2}
\end{figure}

\section{Conclusion and Future Direction}
In this paper, we established a theoretical bound for the low-multilinear-rank tensor recovery from its linear measurements via the Riemannian manifold optimization algorithm. The theoretical recovery guarantee is proved based on the tensor restricted isometry property and the geometry of the low-multilinear-rank tensor manifold. In particular, for an order-$3$ $n\times n\times n$ tensor with a multilinear rank $(r,r,r)$, the number of linear measurements required for an exact recovery is $O(nr^2+r^4)$. This bound of sampling complexity is optimal in $n$, while existing provable tensor recovery approaches usually have a bound unnecessarily large in $n$. 

The robustness analysis of the algorithm to noise is in the scope of future work. It will also be desirable to extend the method for tensor completion and tensor robust principal component analysis models and see whether the algorithm can achieve optimal results. It is also of interest to investigate the algorithm to tensor recovery in frameworks of other low-rank tensor models. 

\bibliographystyle{siam} 
\bibliography{biblio}

\end{document}